\documentclass{svjour3}
\smartqed

\usepackage{amsfonts}
\usepackage{amsmath}
\usepackage{latexsym}
\usepackage{amssymb}
\usepackage{hyperref}
\usepackage[usenames,dvipsnames]{xcolor}
\usepackage{graphics}
\usepackage{graphicx}
\usepackage{algorithm}
\usepackage{algpseudocode}

\usepackage{soul}
\usepackage{cancel}

\newtheorem{ass}{Assumption}

\def\eps{\epsilon}
\def\Re{\mathbb{R}}
\def\one{\mathbf{1}}

\DeclareMathOperator*{\argmin}{argmin}
\DeclareMathOperator*{\vol}{Vol}

\title{Worst-case complexity analysis of derivative-free methods for multi-objective optimization}
\titlerunning{Complexity analysis in DF multiobjective optimization}
\begin{document}

\author{Giampaolo Liuzzi\and Stefano Lucidi}

\institute{Giampaolo Liuzzi, Corresponding author \at
             Sapienza University of Rome\\
              Via Ariosto 25, 00185, Rome (Italy)\\
              liuzzi@diag.uniroma1.it
           \and
              Stefano Lucidi  \at
             Sapienza University of Rome\\
              Via Ariosto 25, 00185, Rome (Italy)\\
              lucidi@diag.uniroma1.it
}

\date{Received: date / Accepted: date}

\maketitle

\newcommand{\DFMOnew}{{\tt DFMOstrong}}
\newcommand{\DFMOlight}{{\tt DFMOlight}}
\newcommand{\DFMOmin}{{\tt DFMOmin}}
\newcommand{\DFMOmax}{{\tt DFMOmax}}
\renewcommand{\thefootnote}{\fnsymbol{footnote}}



{\small
\begin{abstract}
In this work, we are concerned with the worst case complexity analysis of {\em a posteriori} methods for  unconstrained multi-objective optimization problems where objective function values can only be obtained by querying a black box. We present two main algorithms, namely \DFMOnew\ and \DFMOlight\ which are based on a linesearch expansion technique. In particular, \DFMOnew, requires a complete exploration of the points in the current set of non-dominated solutions, whereas \DFMOlight\ only requires the exploration around a single point in the set of non-dominated solutions. For these algorithms, we derive worst case iteration and evaluation complexity results. 
In particular, the complexity results for \DFMOlight\ aligns with those recently proved in
\cite{custodio2021worst} 
for a directional multisearch method. Furthermore, exploiting an expansion technique of the step, we are also able to give further complexity results concerning the number of iterations with a measure of stationarity above a prefixed tolerance.  
\end{abstract}
}




\section{Introduction}
We consider the following unconstrained multi-objective optimization problem
\begin{equation}\label{prob}
 \displaystyle\min_{x\in\Re^n}\ F(x) = (f_1(x),\dots,f_q(x))^\top,
\end{equation}
where $f_i:\Re^n\to \Re$, $i=1,\dots,q$. In the sequel, we assume that the objective functions are 
known only by means of an oracle that only returns their values for any given point $x\in\Re^n$, i.e. first-order information cannot be used and it is impractical or untrustworthy to approximate derivatives. 

In the past decade there has been an increasing interest on the worst-case complexity (WCC) of optimization algorithms (see e.g. \cite{cartis2022evaluation}
and the references therein). 
For single-objective unconstrained optimization, it has been shown \cite{cartis2010complexity} that the steepest descent algorithm has ${\cal O}(\eps^{-2})$ WCC to produce a point with stationarity measure at most equal to $\eps$. Furthermore, the provided bound is tight in the sense that the algorithm can require a number of iterations arbitrarily close to the prescribed WCC. The same WCC result holds for Newton's method, i.e. the Newton's method can be as slow as the steepest descent algorithm. Better WCC bounds can be obtained by using regularized methods and derivatives up to the $p$-th order. In particular, in \cite{birgin2017worst} a WCC of ${\cal O}(\eps^{-(p+1)/p})$ has been proved.  The latter result generalizes the WCC of ${\cal O}(\eps^{-3/2})$ showed for the cubic regularized Newton's method in \cite{cartis2010complexity}. 
When considering derivative-free optimization, in \cite{vicente2013worst} it is proved that directional direct search methods based on sufficient decrease share the same iteration WCC of the steepest descent method, i.e. ${\cal O}(\eps^{-2})$ (the evaluation complexity being ${\cal O}(n\eps^{-2})$ where $n$ is the number of variables). The same WCC bounds and a stronger result have recently been obtained for derivative-free methods based on extrapolation techniques (see \cite{brilli2024worst}). In \cite{cartis2012oracle}, a finite-difference version of the adaptive regularization with cubics algorithm is proposed. It is proved that this special type of derivative-free algorithm has WCC complexity of ${\cal O}(\eps^{-3/2})$ which is thus the same WCC bound obtained for its derivative-based counterpart \cite{cartis2011adaptive_1,cartis2011adaptive_2}.   

In the context of multi-objective optimization, the WCC of algorithms has gained more and more attention in recent years. The first work (to the best of our knowledge) on WCC of a multi-objective method appeared in \cite{fliege2019complexity}. There, it is proved that the multi-objective gradient descent method as the same ${\cal O}(\eps^{-2})$ WCC of the steepest descent algorithm for unconstrained single-objective optimization. It is worth noting that in \cite{fliege2019complexity}, the analyzed algorithm is not intended to reconstruct an approximation of the Pareto front (i.e. producing a set of non-dominated solutions), rather it produces a single point which is (in the limit) a Pareto stationary point of the problem. 
In \cite{calderon2022high} constrained multi-objective problems are considered and a method using high-order models is studied. A WCC of ${\cal O}(\eps^{-(p+1)/p})$ is proved which aligns with analogous results in the single-objective scenario. Again, the method proposed in \cite{calderon2022high} is aimed at producing a single non-dominated point rather than reconstructiog and approximation of the (true) Pareto front. More recently, in \cite{lapucci2024convergence}, the converge and WCC of a wide class of methods for unconstrained multi-objective problems are considered. A general WCC of ${\cal O}(\eps^{-\max\{p_1,2(p_1-p_2)\}})$ is proved where $p_1$ and $p_2$ are connected with the assumptions on the search directions. It is also showed that when $p_2=1$ and $p_1 = 2$ the algorithm has iteration and evaluation complexity of ${\cal O}(\eps^{-2})$. In \cite{lapucci2024effective}, following ideas proposed in \cite{custodio2021worst}, the WCC of algorithms producing set of points rather than a single point per iteration is proved. The WCC bound is ${\cal O}(\eps^{-2q})$ where $q$ is the number of objective functions. When derivatives are not available or impractical to obtain, in
\cite{custodio2021worst} WCC has been studied for the first time for an algorithm producing sets of non-dominated points. In particular, the DMS algorithm has WCC of ${\cal O}(|L(\epsilon)|\epsilon^{-2q})$ number of iterations (${\cal O}(n|L(\epsilon)|\epsilon^{-2q})$ function evaluations) to produce a set of points such that the criticality measure is below $\epsilon$ for at least one point in the list (where $L(\epsilon)$ represents the cardinality of the set of linked sequences between the first unsuccessful iteration and the iteration immediately before the one where the criticality condition is satisfied).    

In this paper we are particularly interested in derivative-free algorithms that produce a set of non-dominated points, namely the so-called {\em a posteriori} methods. We build upon the recent paper \cite{custodio2021worst} and study the convergence and complexity of derivative-free algorithms based on extrapolation techniques. 


\subsection{Our contribution and structure of the paper}
This paper is concerned with the complexity analysis of {\em a posteriori} methods for the solution of problem \eqref{prob}. We present two main algorithms, namely \DFMOnew\ and \DFMOlight. In particular, \DFMOnew, requires a complete exploration of the points in the current set of non-dominated solutions, whereas \DFMOlight\ only require the exploration around a single point in the set of non-dominated solutions. 
We derive worst-case iteration and evaluation complexity bounds for \DFMOnew\ and \DFMOlight. 

Furthermore, we consider two special instances of \DFMOlight, namely \DFMOmin\ and \DFMOmax, and derive worst-case complexity bounds for them both. In the concluding section we compare the bounds obtained for the various algorithms with the aim of highlighting their differences and respective advantages.  



\par\smallskip

The paper is organized as follows.
In section \ref{sec:notation} we introduce the notation that will be used in the paper and we introduce some useful preliminary results.
In section \ref{sec:DFMOstrong} we define algorithm \DFMOnew\ and carry out its theoretical analysis.
In Section \ref{sec:DFMOlight} we describe algorithm \DFMOlight\ and study its worst-case complexity.
Section \ref{sec:special_instances} is devoted to the description and analysis of two special instances of \DFMOlight, namely \DFMOmin\ and \DFMOmax.
Finally, in section \ref{sec:conc} we draw some conclusions.

\section{Notations and preliminary material}\label{sec:notation}

Given a vector $v\in\Re^n$, a subscript will be used to denote either one of its components ($v_i$) or the fact that it is an
element of an infinite sequence of vectors ($v_k$). In case of possible misunderstanding or ambiguities, the $i$th component of a
vector will be denoted by $(v)_i$. 

Given two vectors $u,v\in\Re^n$, we use the following convention for vector equalities and inequalities:
\begin{eqnarray*}
 u = v & \Leftrightarrow & u_i = v_i,\ \forall\ i=1,\dots,n,\\
 u < v     & \Leftrightarrow & u_i < v_i,\ \forall\ i=1,\dots,n,\\
 u \leq v  & \Leftrightarrow & u_i \leq v_i,\ \forall\ i=1,\dots,n,\ \mbox{and}\ u\neq v.
\end{eqnarray*}
Note that $u \geq v$ if and only if $-u \leq -v$.

We denote by $v^j$ the generic $j$th element of a finite set of vectors. Furthermore, vectors $e_1,\dots,e_n$ represent the coordinate unit vectors.
Finally, by $\one$ we denote the vector of all ones of dimension $q$. 
The methods we shall study do not require explicit knowledge of first order derivatives of the objective functions. However, the analysis hinges on the following assumptions concerning differentiability and boundedness of the objective functions 
(see \cite{custodio2021worst}). 
 
\begin{ass}\label{ass:lipschitzcont}
For all $i=1,\dots,q$, the function $f_i$ is continuously differentiable with a Lipschitz continuous gradient with constant $L_i$. Furthermore, let $L_{\max} = \max_{i=1,\dots,q}L_i$.
\end{ass}

\begin{ass}\label{ass:boundedness}
For all $i=1,\dots,q$, and given $x_0\in\Re^n$, it holds that
\[
\begin{split}
&-\infty < f_i^{\min} = \begin{array}[t]{l}\inf\ f_i(x) \\\quad x\in\Re^n: F(x_0)\not\leq F(x),
\end{array}    \\
&+\infty > f_i^{\max} = \begin{array}[t]{l}\sup\ f_i(x) \\\quad x\in\Re^n: F(x_0)\not\leq F(x).
\end{array}    
\end{split}
\]
Let us denote $f_{\max} = (f_1^{\max},\dots,f_q^{\max})^\top\in\Re^q$ and $f_{\min} = (f_1^{\min},\dots,f_q^{\min})^\top\in\Re^q$.
     
\end{ass}
\par\smallskip

When dealing with several objective functions at a time, the concept of Pareto dominance is usually considered in the comparison of two points.

\par\medskip\noindent

\begin{definition}[Pareto dominance]
Given two points, $x,y\in\Re^n$, we say that $x$ dominates $y$ if $F(x)\leq F(y)$.
\end{definition}
\par\smallskip\noindent

The above definition of Pareto dominance is used
to characterize global and local optimality into a multi-objective framework. More specifically, by means of the following
two definitions, we are able to identify a set of non-dominated points (the so called Pareto front or
frontier) which represents the (global or local) optimal solutions of a given multi-objective problem.

\par\medskip\noindent

\begin{definition}[Global Pareto optimality]
A point $x^\star\in\Re^n$ is a global Pareto optimizer of Problem (\ref{prob}) if it does not exist a point $y\in\cal F$ such
that $F(y)\leq F(x^\star)$.
\end{definition}

\begin{definition}[Local Pareto optimality]
A point $x^\star\in\cal F$ is a local Pareto optimizer of Problem (\ref{prob}) if it does not exist a point $y\in\{x\in\Re^n:\ \|x-x^\star\|<\eps\}$ such that $F(y)\leq F(x^\star)$, for some $\eps > 0$.
\end{definition}

\par\medskip\noindent

Now, we can give the following definition of Pareto criticality which is a necessary condition for local/global Pareto optimality.

\begin{definition}[Pareto criticality]
A point $\bar x\in\Re^n$ is Pareto critical for problem \eqref{prob} if for all $d\in\Re^n$, $d\neq 0$, an index $i_d\in\{1,\dots,q\}$ exists such that
\[
\nabla f_{i_d}(\bar x)^\top d \geq 0.
\]
\end{definition}

For a given point $x\in\Re^n$, let us introduce the following Pareto criticality measure
\begin{equation*}
\begin{split}
\mu(x) & = -\min_{\|d\|\leq 1}\max_{i\in\{1,\dots,q\}}\nabla f_i(x)^\top d\\
{\cal F}(x) & = \argmin_{\|d\|\leq 1}\max_{i\in\{1,\dots,q\}}\nabla f_i(x)^\top d
\end{split}
\end{equation*}
\begin{lemma}[{\cite[Lemma 3]{fliege2000steepest}}]
    For a given $x\in\Re^n$, $\mu(x)$ is such that
    \begin{enumerate}
        \item $\mu(x)\geq 0$;
        \item if $x$ is Pareto critical, then $0\in{\cal F}(x)$ and $\mu(x) = 0$;
        \item if $x$ is not Pareto critical, then  $\mu(x) > 0$ and for any $d\in{\cal F}(x)$ we have
        \[
        \nabla f_j(x)^\top d < 0,\qquad\forall\ j\in\{1,\dots,q\};
        \]
        \item the function $x\mapsto \mu(x)$ is continuous;
        \item if $x_k\to\bar x$, $d_k\in {\cal F}(x_k)$ and $d_k\to\bar d$, then $\bar d\in {\cal F}(\bar x)$.
    \end{enumerate}
\end{lemma}
\par\medskip

Given a set of unitary directions $D_k$, we define the following approximation $\mu_{D_k}(x)$ of $\mu(x)$.
\[
\mu_{D_k}(x) = -\min_{d\in D_k}\max_{i=1,\dots,q}\nabla f_i(x)^\top d.
\]

Then, concerning the approximation $\mu_{D_k}(x)$, we introduce the following assumption.

\begin{ass}\label{ass:assumption_mu}
Given sequences of points and sets of directions $\{x_k\}$, $\{D_k\}$, there exists a constant ${\cal C}> 0$ such that
\[
|\mu(x_k)-\mu_{D_k}(x_k)| \leq {\cal C}\mu_{D_k}(x_k)\qquad \text{for all}\ k\in\mathbb{N}.
\]
Furthermore, $|D_k| \leq r = {\cal O}(n)$, for all $k\in\mathbb{N}$.
\end{ass}
\par\smallskip

We require Assumptions \ref{ass:lipschitzcont},  \ref{ass:boundedness} and \ref{ass:assumption_mu} to hold true throughout the paper (hence they will not be explicitly invoked).

\par\smallskip\noindent

\begin{remark}
As noted in \cite{custodio2021worst}, the above assumption implicitly requires that $\mu_{D_k}(x_k)>0$ for all $k$ {such that $x_k$ is not Pareto critical}, i.e. a direction $d\in D_k$ exists which is a descent direction with respect to all the objective functions. This, of course, might not be. In such cases, additional directions could be added to
the set $D_k$ to satisfy the assumption.

Moreover, it is simple to show that in the single objective case the assumption can be easily satisfied by choosing $D_k$ to be sets of unitary directions with a cosine measure uniformly bounded away from zero, i.e. $\sigma(D_k)\geq \sigma > 0$ where 
\[
\sigma(D_k) = \min_{v\in\Re^n}\max_{d\in D_k} \frac{v^\top d}{\|v\|}.
\]
In fact, recalling the definition of cosine measure we can write
\[
\sigma \leq \sigma(D_k) \leq \max_{d\in D_k}- \frac{\nabla f(x_k)^\top d}{\|\nabla f(x_k)\|} = -\min_{d\in D_k} \frac{\nabla f(x_k)^\top d}{\|\nabla f(x_k)\|} =\frac{\mu_{D_k}(x_k)}{\|\nabla f(x_k)\|},
\]
so that, noticing that $\mu(x_k) = \|\nabla f(x_k)\|$ when $q=1$, we have
\[
\mu(x_k) = \|\nabla f(x_k)\| \leq \frac{1}{\sigma}\mu_{D_k}(x_k).
\]
Then, we obtain
\[
|\mu(x_k) -\mu_{D_k}(x_k)| \leq \left(\frac{1}{\sigma} + 1\right)\mu_{D_k}(x_k)
\]
and the assumption is satisfied with ${\cal C} = 1+1/\sigma$.

The above discussion also justifies why we are requiring $|D_k| = r \leq {\cal O}(n)$. In fact, a set $D_k$ with strictly positive cosine measure  has between $n+1$ and $2n$ directions. In the multi-objective case, having a strictly positive cosine measure could still be insufficient to satisfy Assumption \ref{ass:assumption_mu}. We are thus requesting the assumption to be satisfied with a number of directions of the order of $n$, thus possibly more than $2n$.
\end{remark}

\par\medskip

Then, we can give the definition of hypervolume of a set of points with respect to some reference point.
\begin{definition}[{\cite[Definition 3.1]{custodio2021worst}}] The hypervolume indicator for some set $A\subset\Re^q$ and a reference point $\rho \in \Re^q$
that is dominated by all the points in $A$ is defined as:
\[
HI(A) = \vol\{b \in \Re^q : b \leq \rho\ \text{and}\ \exists\ a \in A : a \leq b\} = \vol\left(\bigcup_{a\in A}[a,\rho]\right),
\]
where $[a,\rho] = \{y\in\Re^q:\ a_i\leq y_i\leq \rho_i,\ i=1,\dots,q\}$.
\end{definition}
\par\smallskip

Note that Assumption \ref{ass:boundedness} guarantees that the hypervolume of the set of Pareto points of the problem \eqref{prob} is bounded from above, i.e. a constant $\overline{HI}$ exists such that
\[
HI(F({\cal P})) \leq \overline{HI},
\]
where ${\cal P}\subseteq\Re^n$ is the set of Pareto optimizers of Problem \eqref{prob}.
\par\smallskip

We end this subsection by stating a proposition which will be useful in the analysis of the complexity of the algorithms. In particular, it evidences an interesting property concerning sufficient increase of the Hypervolume indicator.
\par\medskip\noindent
{
\par\medskip\noindent
\begin{proposition}\label{HIincrease_gen}
Let $Y\subseteq\Re^n$ be a set of points, and let the reference point be $\rho = f_{\max}+s\mathbf{1}$ (where $\mathbf{1}$ is the vector of all ones of appropriate dimension and $s>0$ a  sufficiently large scalar parameter). 
\begin{itemize}
\item[i)]
A point $y\in\Re^n$ is not dominated by any point $x\in Y$, i.e. 
\begin{equation}\label{eq:nondomnew1}
F(x) \not\leq F(y),\ F(y)\neq F(x) 
\qquad\forall\ x\in Y,
\end{equation}
if and only if
\[
HI(F(Y\cup\{y\})) - HI(F(Y)) >0. 
\]
\item[ii)] If the point $y\in\Re^n$ and the scalar $\nu>0$ are such that 
\begin{equation*}
F(y) \not > F(x) - \nu\mathbf{1} \qquad\forall\ x\in Y.
\end{equation*}
 then,
\[
HI(F(Y\cup\{y\})) - HI(F(Y)) \geq \nu^q. 
\]
 
\end{itemize}
\end{proposition}   
\begin{proof}
Point (i) is proved in Proposition \ref{HIincrease_gen_new}  {of Appendix A}. As regards proof of point (ii), it follows from the proof of the only if part in Proposition \ref{HIincrease_gen_new} where $\eps = \nu$. 
\end{proof}
}
\par\medskip\noindent

\section{A derivative-free multi-objective algorithm with strong properties}\label{sec:DFMOstrong}\ \par

In the following we report a description of our proposed algorithm for derivative-free multi-objective optimization. 

\DFMOnew\ belongs to the class of {\em a posteriori} methods for multi-objective optimization as it produces a list of non-dominated points. Like other algorithms proposed in the literature (see e.g. DMS \cite{custodio2011direct}, DFMO \cite{liuzzi2016derivative}), the current state of \DFMOnew\ is represented by a list of pairs of points and stepsizes $L_k$ rather than a single point and its associated stepsizes. Then, the aim of an iteration of \DFMOnew\ is to improve the {\em quality} of $L_k$. 
Since we are in a multi-objective context, the quality of $L_k$ can be improved either by adding to $L_k$ one point that dominates some points already in the list; or by increasing the cardinality of $L_k$, that is by adding new non-dominated points to the list. \DFMOnew\ tries to accomplish this two-fold task by performing an exploration around all the points in $L_k$ using the directions belonging to the current set of directions $D_k {= \{d_k^1,\dots,d_k^r\}}$. In particular, let {$x_k\in L_k$ be  the  point currently considered by the algorithm and $X(x_k)$ be the associated} current set of points, i.e. $X(x_k) = \{x_j: (x_j,\tilde\alpha_j)\in\tilde L_k\}$, where $\tilde L_k$ denotes the current list of pairs of points and the corresponding stepsizes {obtained by adding to set $L_k$ the pairs produced by the points (of $L_k$) previously considered by the algorithm at the iteration $k$.} The {\tt Approximate\_optimization} procedure produces $r+1$ sets of points $Y^i(x_k)$, $i=1,\dots,r+1$, such that
\[
X(x_k) = Y^1(x_k) \subseteq Y^2(x_k)\subseteq\dots\subseteq Y^{r+1}(x_k)=Y(x_k).
\]
Each set $Y_k^{i+1}$ is obtained by means of an expansion along $d_k^{i}$ starting from $y_k^{i}$. Specifically, $y_k^1 = x_k$ and, for all $i=1,\dots,r$, $y_k^{i+1}$ is either $y_k^i+\alpha_k^id_k^i$ or $y_k^i$ depending on whether new sufficiently non-dominated points have been added to $Y^{i+1}(x_k)$ or not.
During the expansion, new points are added to the list provided that they are sufficiently non-dominated by any point already in the list, i.e. 
\[
F(y_k^{i}+\alpha d_k^{i}) \not > F(x_j) -\gamma(\alpha)^2\mathbf{1},\qquad \forall\ x_j\in Y^{i+1}(x_k).
\]

The expansion cycle stops when the tentative point is sufficiently dominated by one point belonging to the current set $Y^{i+1}(x_k)$, i.e. when
\[
F(y_k^{i}+\beta d_k^{i})  > F(x_j) -\gamma(\beta)^2\mathbf{1},\qquad \text{for at least an}\ x_j\in Y^{i+1}(x_k);
\]
If the set of points $X(x_k)$ at the beginning of the exploration around $x_k$ is equal to the set of points $Y(x_k)$ produced by the {\tt Approximate\_optimization} procedure, i.e. when $Y(x_k)=X(x_k)$, the exploration fails and the step sizes associated with $x_k$ are decreased by a constant factor. Otherwise, the points in $Y(x_k)$ are added to a temporary list of produced points each one associated with the lastly produced stepsizes.   



At the end of the iteration, namely after exploration of all the points in $L_k$, the new list $L_{k+1}$ is generated by choosing the pairs corresponding to the non-dominated points in the temporary list $\tilde L_k$. An iteration $k$ of the algorithm, i.e. the exploration of all the points in $X_k$, is {\em successful} when $X_k\neq X_{k+1}$ (it is unsuccessful when $X_k = X_{k+1}$).


\begin{algorithm}[ht!]
\caption{\DFMOnew}
\begin{algorithmic}[1]
\State   {\bf data.} $\theta \in (0,1)$, $r\in\mathbb{N}$, $r>0$, a sequence $\{D_k\}$, $D_k\in \Re^{n\times r}$ for all $k$, $L_0=\{(x_j,\tilde \alpha_j),\ x_j \in \Re^n,\ \tilde\alpha_j\in \Re^r_{+},\ j = 1,\dots,|L_0| \}$, $\bar\Delta_0 \leftarrow \displaystyle\max_{(x_j,\tilde\alpha_j)\in  L_0}\displaystyle\max_{i=1,\dots,r}\tilde\alpha_j^i$

\For{$k=0,1,\dots$}
\State set $\tilde L_k \leftarrow L_k$, $X_k\leftarrow\{x_j: (x_j,\tilde\alpha_j)\in L_k  \}$
\ForAll{$(x_k,\tilde\alpha_k)\in L_k$}
\State set $X(x_k)\leftarrow\{x_j: (x_j,\tilde\alpha_j)\in \tilde L_k  \}$

\State compute $Y(x_k)$ and $\alpha_k$ by the 
\par\phantom{7:}\qquad \texttt{Approximate Optimization} $(x_k,\tilde \alpha_k,X(x_k),c\bar\Delta_k,D_k)$ 

\If{$X(x_k) = Y(x_k)$}
\State set $\tilde\alpha_{k+1}^i \leftarrow \theta\alpha_k^i$ for all $i=1,\dots,r$

\State set $\tilde L_k \leftarrow (\tilde L_k\setminus\{(x_k,\tilde\alpha_k)\})\cup \{(x_k,\tilde\alpha_{k+1})\}$

\Else

\State set $\tilde\alpha_{k+1}^i \leftarrow \alpha_k^i$ for all $i=1,\dots,r$

\ForAll{$y\in Y(x_k)$}

\State set $\tilde L_k \leftarrow \tilde L_k\cup \{(y,\tilde\alpha_{k+1})\}$

\EndFor
\EndIf
\EndFor

\State set $\bar\Delta_{k+1} \leftarrow \displaystyle\max_{(x_j,\tilde\alpha_j)\in \tilde L_k}\displaystyle\max_{i=1,\dots,r}\tilde\alpha_j^i$

\State set $L_{k+1}$ the non-dominated points in  $\tilde L_k$

\EndFor

\end{algorithmic}
\end{algorithm}

\begin{algorithm}
\caption{{\tt Approximate Optimization} ($x,\tilde\alpha,X(x),\bar\nu,D$)}
\begin{algorithmic}[1]
\State {\bf data}. $\delta\in(0,1)$, $\gamma\in(0,1)$
\State set $y^1 \leftarrow x$, $Y^1(x)\leftarrow X(x)$
\Statex {\em (explore the directions in $D=\{d^1,\dots,d^r\}$)}
\For{$i=1,\dots,r$ }
\State select $d^i\in D$
\State set $\nu^i \leftarrow \max\{\tilde\alpha^i, \bar\nu\}$, $Y^{i+1}(x)\leftarrow Y^i(x)$, $y^{i+1}\leftarrow y^i$, $\alpha^i \leftarrow \nu^i$, $\beta \leftarrow \nu^i$
\While{$\Big(\!F(y^i+\beta d^i) \not >  F(x_j) -\gamma \beta^2\one\!$\Big) \ $\forall\ x_j\in Y^{i+1}(x)$}
\State set $\alpha \leftarrow \beta$, $Y^{i+1}(x)\leftarrow Y^{i+1}(x) \cup \{y^i+\alpha d^i\}$, $\beta \leftarrow \alpha/\delta$

\State set $y^{i+1} \leftarrow y^i+\alpha d^i$, $\alpha^i \leftarrow \alpha$

\EndWhile
\EndFor

\State {\bf return} $Y^{r+1}(x)$ and $\alpha$
\end{algorithmic}
\end{algorithm}

\label{sec:convergence}

\subsection{Preliminary results}\label{sec:stationarity}
\ \par
In this subsection we report some preliminary results that will be used in the worst-case analysis. Specifically, we have three main results:
\begin{enumerate}
    \item the first one relates the measure of stationarity of every point $x_k$ in the list $L_k$ with the maximum of its associated stepsizes;
    \item the second result is about hypervolume improvement; in particular, we prove that the hypervolume of the list of points in successful iterations sufficiently improves by a quantity proportional to the maximum stepsize of all points to the power of $q$;
    \item finally, we show that \DFMOnew\ always sufficiently improves the value of an auxiliary function $\Phi_k$, which is used in the worst-case complexity analysis and which can be used to measure progress of the algorithm.
\end{enumerate}

First of all, we show that the {\tt Approximate Optimization} procedure is well-defined, i.e. it cannot infinitely cycle in the inner {\tt while} loop.
\par\smallskip

\begin{lemma}
The {\tt Approximate Optimization} procedure is well-defined, i.e. the inner {\tt while} loop cannot infinitely cycle for every $i\in\{1,\dots,r\}$.     
\end{lemma}
\begin{proof}
By contradiction, let us suppose that and index $\bar\imath\in\{1,\dots,r\}$ exists such that the {\tt while} loop infinitely cycles. If this is the case, for all $j=1,2,\dots$, we have
\[
F(y^{\bar\imath}+\alpha/\delta^j d^{\bar\imath}) \not >  F(x_j) -\gamma (\alpha/\delta^j)^2\one \quad \forall\ x_j\in Y^{{\bar\imath}+1}(x)
\]
and the procedure considers the set $Y^{\bar\imath+1}(x)\cup\{y^{\bar\imath}+\alpha/\delta^j d^{\bar\imath}\}$. By Proposition \ref{HIincrease_gen}, we can write, for all $j=1,2,\dots$,
\[
HI(F(Y^{\bar\imath+1}(x)\cup\{y^{\bar\imath}+\alpha/\delta^j d^{\bar\imath}\})) - HI (F(Y^{\bar\imath+1}(x))) \geq \gamma(\alpha/\delta^j)^{2q}.
\]
The above relation, along with the definition of the \texttt{Approximate Optimization} procedure in which $Y^{\bar\imath+1}(x)$ is set equal to $Y^{\bar\imath+1}(x)\cup\{y^{\bar\imath}+\alpha/\delta^j d^{\bar\imath}\}$, contrasts with the hypervolume being upper bounded by $\vol([f_{\min},\rho])$ under Assumption \ref{ass:boundedness},  
thus concluding the proof.
\end{proof}
\par\smallskip

In the following proposition, by exploiting the {\tt Approximate optimization} procedure and in particular the expansion technique therein adopted, we can bound the approximate criticality measure, $\mu_{D}(x)$, to the maximum of the stepsizes produced by the {\tt Approximate optimization}.

To better understand the following theoretical result, it is worth noting that for each pair $(x_k, \tilde\alpha_k)\in L_k$ the algorithm \DFMOnew\  produces a set of points $Y_k$ and a vector $\alpha_k$ of stepsizes. 
\par\smallskip


\begin{proposition}\label{prop:boundmu}
Let $x_k\in X_k$ and $\tilde\alpha^i_{k+1}$, $i=1,\dots,r$, be the stepsizes produced by the {\tt Approximate optimization} procedure starting from $x_k$. Let  
\[
\Delta_{k+1} = \max_{i=1,\dots,r}\{\tilde\alpha^i_{k+1}\}.
\]
Then, the following bound on $\mu_{D_k}(x_k)$ holds true for all iterations $k$ and all $x_k\in X_k$
\[
\mu_{D_k}(x_k) \leq \left\{\begin{array}{ll}
     \left(\displaystyle\frac{2(L^{max} + \gamma)}{\delta(1-\delta)}   + L^{max}\sqrt{n}\right)\Delta_{k+1} = (c_1+L^{max}\sqrt{n})\Delta_{k+1}& \text{if}\ Y(x_k)\neq X(x_k) \\
     \left(\displaystyle\frac{L^{max} + 2\gamma}{2\theta} \right)\Delta_{k+1} = c_2\Delta_{k+1}& \text{if}\ Y(x_k) = X(x_k). 
\end{array}\right.
\]
\end{proposition}

\begin{proof}
Let us consider a point $x_k\in X_k$ and let us distinguish the following cases: (a) successful exploration ($Y(x_k)\neq X(x_k)$, i.e. $y_k^{r+1}\neq x_k$); (b) unsuccessful exploration ($Y(x_k) = X(x_k)$, i.e. $y_k^{r+1} = x_k$).

\begin{itemize}
    \item[(a)] 
    {For every $i=1,\dots,n$,} {we consider the two subcases: $Y^i(x_k) = Y^{i+1}(x_k)$, i.e. $y_k^i=y_k^{i+1}$,  and $Y^i(x_k)\neq Y^{i+1}(x_k)$, i.e. $y_k^i\neq y_k^{i+1}$.  
    
    If $Y^i(x_k)=Y^{i+1}(x_k)$ then $\hat x\in Y^i(x_k)$ exists such that 
    \begin{equation}\label{app1}
    F(y_k^i+\ {\tilde\alpha_{k+1}^i}d_k^i) >  F(\hat x) -\gamma ({\tilde\alpha_{k+1}^i})^2\one
    \end{equation}
    Since  $y_k^i\in Y_k^i(x_k)$ and  $y_k^i$ is not dominated by any point in $Y_k^i(x_k)$  then an index $j(i)$ exists such that
    $
      f_{j(i)}(y_k^i)<  f_{j(i)}(\hat x).
    $ 
    Therefore, recalling (\ref{app1}),
     \[
      f_{j(i)}(y_k^i + {\tilde\alpha_{k+1}^i} d_k^i) > f_{j(i)}(y_k^i) -\gamma({\tilde\alpha_{k+1}^i})^2
    \] 
Then, we can write
    \[
    \begin{split}
        0 &\leq f_{j(i)}(y_k^i + {\tilde\alpha_{k+1}^i} d_k^i) -f_{j(i)}(y_k^i) +\gamma({\tilde\alpha_{k+1}^i})^2  \\
        =& \int_0^1 \nabla f_{j(i)}(y_k^i + t{\tilde\alpha_{k+1}^i} d_k^i)^\top  {\tilde\alpha_{k+1}^i} d_k^i dt + \gamma({\tilde\alpha_{k+1}^i})^2
    \end{split}
    \]
    Adding $-{\tilde\alpha_{k+1}^i}\nabla f_{j(i)}(y_k^i)^\top d_k^i$ to both sides,
    \[
    \begin{split}
        -{\tilde\alpha_{k+1}^i}\nabla f_{j(i)}(y_k^i)^\top d_k^i &\leq \int_0^1 (\nabla f_{j(i)}(y_k^i + t{\tilde\alpha_{k+1}^i} d_k^i) - \nabla f_{j(i)}(y_k^i))^\top {\tilde\alpha_{k+1}^i} d_k^i dt + \gamma({\tilde\alpha_{k+1}^i})^2 \\
        & \leq ({\tilde\alpha_{k+1}^i})^2 \frac{L^{max}}{2} + \gamma ({\tilde\alpha_{k+1}^i})^2
    \end{split}
    \]
    So that 
    \[
    -\nabla f_{j(i)}(y_k^i)^\top d_k^i \leq  \left(\frac{L^{max}}{2} + \gamma \right) {\tilde\alpha_{k+1}^i}
    \]
    \[
    \begin{split}
    -\nabla f_{j(i)}(x_k)^\top d_k^i & = (-\nabla f_{j(i)}(x_k) + \nabla f_{j(i)}(y_k^i) - \nabla f_{j(i)}(y_k^i))^\top d_k^i \\
    & \leq  \left(\frac{L^{max}}{2} + \gamma \right) {\tilde\alpha_{k+1}^i} + L^{max}\|x_k-y_k^i\|\\
    & \leq \left(\frac{L^{max}}{2} + \gamma \right){\max_{i=1,\dots,r}\{\tilde\alpha_{k+1}^i\}} + L^{max}\sqrt{n}\max_{i=1,\dots,r}\{\tilde\alpha_{k+1}^i\}\\
    & \leq \left({L^{max} + \gamma}  + L^{max}\sqrt{n}\right)\max_{i=1,\dots,r}\{\tilde\alpha_{k+1}^i\}.
    \end{split}
    \]
If $Y^i(x_k) \not=Y^{i+1}(x_k)$ then $\tilde  x\in Y^{i+1}(x_k)$ exists such that
    \begin{equation}\label{app2}
    F\left(y_k^i+\ \frac{\tilde\alpha_{k+1}^i}{\delta }d_k^i\right) >  F(\tilde x) -\gamma \left(\frac{\tilde\alpha_{k+1}^i}{\delta }\right)^2\!\!\one.
    \end{equation}
    Since  $y_k^i+\tilde\alpha_{k+1}^id_k^i\in Y^{i+1}(x_k)$ and  $y_k^i+\tilde\alpha_{k+1}^id_k^i$ is not dominated by any point in $Y^{i+1}(x_k)$  then an index $j(i)$ exists such that
    $$
      f_{j(i)}(y_k^i+\tilde\alpha_{k+1}^id_k^i)<  f_{j(i)}(\tilde x).
    $$ Therefore, recalling (\ref{app2}),
     \[
      f_{j(i)}(y_k^i + \frac{\tilde\alpha_{k+1}^i}{\delta } d_k^i) > f_{j(i)}(y_k^i+\tilde\alpha_{k+1}^id_k^i) -\gamma(\frac{\tilde\alpha_{k+1}^i}{\delta })^2
    \] 
Then, we can write
    \[
    \begin{split}
        0 &\leq f_{j(i)}(y_k^i + \frac{\tilde\alpha_{k+1}^i}{\delta } d_k^i) -f_{j(i)}(y_k^i+\tilde\alpha_{k+1}^id_k^i) +\gamma(\frac{\tilde\alpha_{k+1}^i}{\delta})^2  \\
                =& \int_0^1 \nabla f_{j(i)}(y_k^i+(1 + t\frac{1-\delta}{\delta })\tilde\alpha_{k+1}^i d_k^i)^\top 
                  \frac{1-\delta}{\delta } \tilde\alpha_{k+1}^id_k^i dt
                +\gamma(\frac{\tilde\alpha_{k+1}^i}{\delta })^2
    \end{split}
    \]
\par
\par
    Adding $-\tilde\alpha_{k+1}^i\frac{1-\delta}{\delta}\nabla f_{j(i)}(y_k^i)^\top d_k^i$ to both sides,
    \[
    \begin{split}
& -\tilde\alpha_{k+1}^i\frac{1-\delta}{\delta}\nabla f_{j(i)}(y_k^i)^\top d_k^i \\
&\quad \leq \int_0^1 (\nabla f_{j(i)}(y_k^i+(1 + t\frac{1-\delta}{\delta })\tilde\alpha_{k+1}^i d_k^i)-\nabla f_{j(i)}(y_k^i))^\top 
                  \frac{1-\delta}{\delta } \tilde\alpha_{k+1}^id_k^i dt
                +\gamma(\frac{\tilde\alpha_{k+1}^i}{\delta })^2 \\
    & \quad \leq \frac{1-\delta}{\delta } \tilde\alpha_{k+1}^i\int_0^1 L^{max} (1+ t\frac{1-\delta}{\delta }) \tilde\alpha_{k+1}^idt +\gamma(\frac{\tilde\alpha_{k+1}^i}{\delta })^2\\
        & \quad = 
    \frac{1-\delta}{\delta } (\tilde\alpha_{k+1}^i)^2L^{max}(\frac{1+\delta}{2\delta}) +\gamma(\frac{\tilde\alpha_{k+1}^i}{\delta })^2
    \end{split}
    \]
    So that 
    \[
    -\nabla f_{j(i)}(y_k^i)^\top d_k^i \leq  \tilde\alpha_{k+1}^iL^{max}(\frac{1+\delta}{2\delta}) +\gamma\frac{\tilde\alpha_{k+1}^i}{\delta(1-\delta) } \leq \tilde\alpha_{k+1}^i\frac{1+\delta}{\delta(1-\delta)}\left(L^{max} +\gamma \right)
    \]
and we obtain
    \[
    \begin{split}
    -\nabla f_{j(i)}(x_k)^\top d_k^i & = (-\nabla f_{j(i)}(x_k) + \nabla f_{j(i)}(y_k^i) - \nabla f_{j(i)}(y_k^i))^\top d_k^i \\
    & \leq  \tilde\alpha_{k+1}^i\frac{1+\delta}{\delta(1-\delta)}\left(L^{max} +\gamma \right)+ L^{max}\|x_k-y_k^i\|\\
    & \leq \frac{1+\delta}{\delta(1-\delta)}\left(L^{max} +\gamma \right) \max_{i=1,\dots,r}\{\tilde\alpha_{k+1}^i\}+ L^{max}\sqrt{n}\max_{i=1,\dots,r}\{\tilde\alpha_{k+1}^i\}\\
    & \leq \left(\frac{2}{\delta(1-\delta)}\left(L^{max} +\gamma \right) + L^{max}\sqrt{n}\right)\max_{i=1,\dots,r}\{\tilde\alpha_{k+1}^i\}.
    \end{split}
    \]
     }

    Then, considering the definition of $\mu_{D_k}(x)$, we can write
    \begin{equation}\label{boundmu_succ}
    \mu_{D_k}(x_k) \leq \left(\frac{2(L^{max} + \gamma)}{\delta(1-\delta)}   + L^{max}\sqrt{n}\right)\max_{i=1,\dots,r}\{\tilde\alpha_{k+1}^i\}.
    \end{equation}

    \item[(b)] In this case we have $Y(x_k) = X(x_k)$ so that it results $Y^i(x_k) = X(x_k)$, for all $i=1,\dots,r+1$. Then, by reasoning as above we can prove
    \begin{equation}\label{boundmu_fall}
    \mu_{D_k}(x_k) \leq \left(\frac{L^{max} + 2\gamma}{2\theta} \right)\max_{i=1,\dots,r}\{\tilde\alpha_{k+1}^i\}.
    \end{equation}
\end{itemize}
Finally, \eqref{boundmu_succ} and \eqref{boundmu_fall} achieves the proof.
\end{proof}
\par\medskip

In the following proposition, we show that the sequence of hypervolume values $\{HI(F(X_k))\}$ is non-decreasing. In particular, we show that for every successfull iteration of the algorithm the hypervolume strictly increases of a quantity which is related to the maximum of the maximum stepsizes which is defined in the algorithm, i.e.
\[
\bar\Delta_{k+1} \stackrel{\triangle}{=} \max_{(x_j,\tilde\alpha_j)\in \tilde L_k} \max_{i=1,\dots,r}\tilde\alpha_j^i
\]

\begin{proposition}\label{HIincrease}
Let the reference point be $\rho = f_{\max}+s\mathbf{1}$ (with $s> 0$ sufficiently large scalar paraneter). Then,
    \DFMOnew\ is such that, for a successful iteration (namely when $X_k\neq X_{k+1}$), we have
    \[
    HI(F(X_{k+1})) - HI(F(X_k)) \geq \left\{\begin{array}{ll}
         \left(\gamma\ c^2\bar\Delta_{k+1}^2\right)^q & \ \text{if}\ \bar\Delta_k = \bar\Delta_{k+1} \\
         \left(\gamma\ \bar\Delta_{k+1}^2\right)^q&  \ \text{if}\ \bar\Delta_k < \bar\Delta_{k+1} 
    \end{array}\right.
    \] 
\end{proposition}
\begin{proof}
    For a successful iteration, only the following two cases can occur. 
    \begin{itemize}
    \item $\bar\Delta_k = \bar\Delta_{k+1}$. In this case, an $x_k\in X_k$ exists such that the {\tt Approximate optimization} procedure starting from $x_k$ produced at least a point $y_k^j$, with $j\in\{2,\dots,r+1\}$. This, in turn implies that $y_k^j$ is such that
    \[
     F(y_k^j) \not > F(x_t) - \gamma(\nu_k^{j-1})^2\mathbf{1},\quad\forall\ x_t\in Y^j(x_k).
    \]
    Then, recalling Proposition \ref{HIincrease_gen} and that $HI(F(X_k))\leq HI(F(Y^j(x_k)))$ and $HI(F(X_{k+1}))\geq HI(F(Y^j(x_k)\cup\{y_k^j\}))$, we can write
    \[
    HI(F(X_{k+1})) - HI(F(X_k)) \geq HI(F(Y^j(x_k)\cup\{y_k^j\})) - HI(F(Y^j(x_k))) \geq (\gamma(\nu_k^{j-1})^2)^q
    \]
    which, recalling that $\nu_k^{j-1}\geq c\bar\Delta_k = c\bar\Delta_{k+1}$ yields
    \[
    HI(F(X_{k+1})) - HI(F(X_k)) \geq(\gamma c^2\bar\Delta_{k+1}^2)^q.
    \]

    \item $\bar\Delta_{k+1} > \bar\Delta_k$. In this case, an $x_k\in X_k$ exists such that the {\tt Approximate optimization} procedure starting from $x_k$ produced at least a point $y_k^j$ with $j\in\{2,\dots,r+1\}$ using a step size $\alpha_k^{j-1}$ such that 
    \[
    \alpha_k^{j-1} = \bar\Delta_{k+1}
    \]
    and we can write
    \[
     F(y_k^j) \not > F(x_t) - \gamma(\bar\Delta_{k+1})^2\mathbf{1},\quad\forall\ x_t\in Y^j(x_k).
    \]
    Then, recalling again Proposition \ref{HIincrease_gen} and that $HI(F(X_k))\leq HI(F(Y^j(x_k)))$ and $HI(F(X_{k+1}))\geq HI(F(Y^j(x_k)\cup\{y_k^j\}))$, we can write
    \[
    HI(F(X_{k+1})) - HI(F(X_k)) \geq HI(F(Y^j(x_k)\cup\{y_k^j\})) - HI(F(Y^j(x_k))) \geq (\gamma\bar\Delta_{k+1}^2)^q.
    \]
    \end{itemize}
    Then, the proof is concluded.
\end{proof}

\par\medskip

To better analyze the behavior of Algorithm \DFMOnew, let us define the {\em cost function}
\[
\Phi_k = -HI(F(X_k)) + \eta\bar\Delta_k^{2q}
\]
for a given parameter $\eta$ with $0 < \eta < \gamma^q$. Note that, under the stated assumptions, the hypervolume is bounded from above, i.e.
\[
HI(F(X_k)) \leq \vol\left([f_{\min},r]\right) = \overline{HI},
\]
hence, for all $k$, the cost function is bounded from below, that is
\[
-\overline{HI} \leq \Phi_k.
\]
In the next proposition, we show that the cost function always sufficiently decreases of a quantity which is related to $\bar\Delta_{k+1}$.

\begin{proposition}\label{phidecrease}\label{Jeps}
For every iteration of \DFMOnew, it holds that,
\[
\Phi_{k+1} - \Phi_k \leq -\tilde c (\bar\Delta_{k+1}^{2})^q
\]
with 
\begin{equation}\label{ctilde}
\tilde c = \min\left\{\eta\frac{1-\theta^{2q}}{\theta^{2q}}, \gamma^q c^{2q}, \gamma^q-\eta\right\} > 0
\end{equation}
\end{proposition}
\begin{proof}
For every iteration $k$, only following cases can occur.
\begin{itemize}
    \item $X_{k+1} = X_k$. In this case, we have $\bar\Delta_{k+1} \leq \theta\bar\Delta_k$. Then, we can write
    \[
    \Phi_{k+1} - \Phi_k = \eta(\bar\Delta_{k+1})^{2q} - \eta(\bar\Delta_k)^{2q} \leq \eta\left(1 - \frac{1}{\theta^{2q}}\right)(\bar\Delta_{k+1})^{2q} = -\eta \frac{1 - \theta^{2q}}{\theta^{2q}}(\bar\Delta_{k+1})^{2q}.
    \]
    \item $X_{k+1}\neq X_k$ and $\bar\Delta_{k+1} \leq \bar\Delta_k$. By proposition \ref{HIincrease}, in this case we can write,
    \[
    \Phi_{k+1} - \Phi_k \leq -(\gamma c^2\bar\Delta_{k+1}^2)^q + \eta(\bar\Delta_{k+1}^{2q} - \bar\Delta_k^{2q})\leq -(\gamma c^2\bar\Delta_{k+1}^2)^q.
    \]
    \item $X_{k+1}\neq X_k$ and $\bar\Delta_{k+1} > \bar\Delta_k$. In this case, again by proposition \ref{HIincrease}, we get
    \[
    \Phi_{k+1} - \Phi_k \leq -(\gamma \bar\Delta_{k+1}^2)^q + \eta\bar\Delta_{k+1}^{2q} - \eta\bar\Delta_k^{2q}\leq -(\gamma^q-\eta)\bar\Delta_{k+1}^{2q}.
    \]
\end{itemize}
Then, the proof is concluded.
\end{proof}\par\noindent\smallskip

\subsection{Worst-case complexity of \DFMOnew}\ \par
In this section, we give complexity bounds for the algorithm \DFMOnew. To this aim, let us introduce the following criticality function. In particular, given a set of points $X\subseteq\Re^n$, let
\begin{equation}\label{Gamma}
\Gamma(X) = \max_{x\in X}\mu(x),
\end{equation}
i.e. $\Gamma(X)$ associates to every set of points $X$ a measure of stationarity as the worst stationarity measure of the points belonging to the set $X$ itself.

Then, we provide worst-case bounds for the number of iterations, number of function evaluations required by the algorithm to achieve a set $X_k$ such $\Gamma(X_k)\leq\eps$, i.e. such that all the points $x_i\in X_k$ have a criticality measure $\mu(x_i)$ at most equal to the prefixed threshold $\eps$. Furthermore, we are also able to bound the total number of iterations such that $\Gamma(X_k)>\eps$, i.e. such that (at least) a point $x_i\in X_k$ exists with a criticality measure $\mu(x_i)>\eps$. 

\begin{proposition}\label{prop:Kstrong}
For any $\epsilon > 0$, consider the following subset of indices:
	\begin{equation}\label{subset}
K_{\eps}=\left\{\ k\in\{0,1,\ldots\}:\ \Gamma(X_k) > \epsilon  \ \right\}.
\end{equation}	
 Then,
\[
|K_{\eps}| \leq \left\lfloor
\frac{(\Phi_0+\overline{HI})(\hat c({\cal C}+1))^{2q}}{\tilde c}\epsilon^{-2q}
\right\rfloor\leq {\cal O}(n^q\,\epsilon^{-2q}).
\]
\end{proposition}
\begin{proof}
By the stated hypothesis and Assumption \ref{ass:assumption_mu}, 
for all iterations $k = 0,1,\dots$, and for all  points $x_i\in X_k$, we have:
\[
\epsilon < \mu(x_i)  \leq ({\cal C}+1)\mu_{D_k}(x_i),
\]
from which:
\[
\epsilon < \max_{ x_i\in X_k} \mu(x_i)  \leq ({\cal C}+1)\max_{ x_i\in X_k} \mu_{D_k}(x_i),
\] Then, from proposition \ref{prop:boundmu}, we have
\begin{equation}\label{boundbardelta1}
\frac{\epsilon}{{\cal C}+1} < \max_{ x_i\in X_k} \mu_{D_k}(x_i) \leq \hat c\bar\Delta_{k+1}.
\end{equation}
where $\hat c = \max\{c_1+L^{max}\sqrt{n},c_2\}\leq {\cal O}(n^{1/2})$ and $c_1,c_2$ are defined in proposition \ref{prop:boundmu}.
Then, from proposition \ref{phidecrease}, we have
\begin{equation}\label{complKeps}
\Phi_0- \Phi_{k+1} \geq \sum_{\tilde k=0}^k\tilde c(\bar\Delta_{{\tilde k}+1}^2)^q,
\end{equation}
The sequence $\{\Phi_k\}$ is bounded from below hence  $\Phi^*$ exists such that:
$$\lim_{k\to\infty}\Phi_k=\Phi^*\ge-\overline{HI}.$$
Taking the limit for $k\to\infty$ in (\ref{complKeps}) we obtain:

\[
\Phi_0+\overline{HI} \geq \sum_{ k=1}^\infty\tilde c(\bar\Delta_{{ k}+1}^2)^q \geq \sum_{ k\in K_{\eps}} \tilde c(\bar\Delta_{{ k}+1}^2)^q\geq |K_{\epsilon}|\tilde c \frac{\epsilon^{2q}}{\hat c^{2q}({\cal C}+1)^{2q}}
\]
Then, we obtain
\[
|K_{\epsilon}| \leq \left\lfloor
\frac{(\Phi_0+\overline{HI})(\hat c({\cal C}+1))^{2q}}{\tilde c}\epsilon^{-2q}
\right\rfloor\leq {\cal O}(n^q\eps^{-2q}),
\]
thus concluding the proof.
\end{proof}
\par\smallskip

\begin{remark}\label{bound_iter_strong}
 For any $\eps >0$, let $j_{\epsilon}$ be the first iteration of algorithm \DFMOnew\ such that $\Gamma(X_{j_{\eps}}) \leq\eps$, i.e.  $\mu(x_i) \leq\epsilon$ for all $x_i\in X_{j_{\epsilon}}$. Noting that 
$j_{\epsilon} \leq |K_\epsilon|$,
 we also have that
 \[
 j_\epsilon \leq {\cal O}(n^q\epsilon^{-2q}).
 \]
\end{remark}
\par\smallskip

Exploiting Proposition \ref{prop:Kstrong} we can immediately give the following corollary which states asymptotical convergence properties of the sequence of sets $\{X_k\}$ produced by the algorithm \DFMOnew.
\par\medskip

\begin{corollary}
Let $\{X_k\}$ be the sequence of sets of points produced by algorithm \DFMOnew. Then:
\begin{enumerate}
    \item[(i)] $\lim_{k\to\infty}\Gamma(X_k) = 0;$
    \item[(ii)] if $\{x_k\}$ is a sequence such that $x_k\in X_k$ for all $k$, $\lim_{k\to\infty}\mu(x_k) = 0.$ 
\end{enumerate}
\end{corollary}
\begin{proof}
    The proof of point (i) immediately follows from the definition of limit and Proposition \ref{prop:Kstrong}.

    Concerning point (ii), for every iteration $k$, by definition of $\Gamma$, we can write
    \[
    0 \leq \mu(x_k) \leq \Gamma(X_k).
    \]
    Hence, the proof is concluded by recalling point (i) above.
\end{proof}
\par\medskip

\begin{proposition}\label{prop:NF_DFMOnew}
Given any $\epsilon > 0$, let $j_{\epsilon} \ge 1$ be the first iteration such that $\Gamma(X_{j_{\eps}})\leq\eps$, i.e. $\mu(x_{i})\leq\epsilon$ for all $x_i\in X_{j_{\epsilon}}$.
Then, denoting by $Nf_{j_{\epsilon}}$ the number of function evaluations required by Algorithm \DFMOnew\ up to iteration $j_{\epsilon}$, we have that
$Nf_{j_{\epsilon}} \leq {\cal O}(n\,L(\epsilon)\,\epsilon^{-2q})$, where $L(\eps) = \max_{k=0,\dots,j_\eps}|L_k|$. In particular,
\[
Nf_{j_{\epsilon}} \leq {{\cal O}(n)} |L(\epsilon)| \left\lfloor\frac{(\overline{HI} + \Phi_0)(\hat c({\cal C}+1))^{2q}}{\tilde c}\epsilon^{-2q}\right\rfloor + \left\lfloor\frac{(\overline{HI} - HI_0)(\hat c({\cal C}+1))^{2q}}{\gamma^qc^{2q}}\epsilon^{-2q}\right\rfloor,
\]
where $\tilde c$  and $c_2$ are given by~\eqref{ctilde} and in proposition \eqref{prop:boundmu}, respectively.
\end{proposition}

\begin{proof}
First, let us partition the set of iteration indices $\{0,\dots,j_{\epsilon}-1\}$ into $\mathcal S_{j_{\epsilon}}$ and $\mathcal U_{j_{\epsilon}}$ such that
\[
k \in \mathcal S_{j_{\epsilon}} \, \Leftrightarrow \,  L_k \not= L_{k-1}, \quad k\in \mathcal U_{j_{\epsilon}} \, \Leftrightarrow \,  L_k = L_{k-1}, \quad \mathcal S_{j_{\epsilon}} \cup \mathcal U_{j_{\epsilon}} = \{0,\dots,j_{\epsilon}-1\},
\]
that is, $\mathcal S_{j_{\epsilon}}$ and $\mathcal U_{j_{\epsilon}}$ contain the successful and unsuccessful iterations up to $j_{\epsilon}-1$, respectively.



When the algorithm explores a new point, the latter can either succeed to increase the hypervolume or fail to do so. Let us then define $Nf_{j_{\epsilon}}^{\mathcal S}$ as the total number of function evaluations related to points which succeed to increase the hypervolume up to iteration $j_{\epsilon}$. Note that, at each iteration, the maximum number of function evaluations related to points which fail to increase the hypervolume is $r|L_k|$.
From the instructions of the algorithm, we can write
\begin{equation}\label{NF}
Nf_{j_{\epsilon}} \leq r\ j_{\epsilon}\ \max_{i=0,\dots,j_{\epsilon}}|L_k| + Nf_{j_{\epsilon}}^{\mathcal S} \le {\cal O}(n) |L(\epsilon)| \left\lfloor\frac{(\overline{HI} + \Phi_0)(\hat c({\cal C}+1))^{2q}}{\tilde c}\epsilon^{-2q}\right\rfloor + Nf_{j_{\epsilon}}^{\mathcal S},
\end{equation}
where the last inequality follows from Proposition~\ref{prop:Kstrong} and Remark \ref{bound_iter_strong}.
Now, let us consider any iteration $k < j_{\epsilon}$ and any index $i \in \{1,\ldots,n\}$
such that the line search succeeds to insert a new point in $\hat L$. This implies the following relation involving the hypervolume.
\begin{equation}\label{eq:bound_diff_f}
HI(F(Y^{i+1}(x_k)\cup\{y_k^i+\alpha_k^i d_k^i\})) - HI(F(Y^{i+1}(x_k))) \geq (\gamma (\alpha_k^i)^2)^q \geq \left(\gamma c^2\bar\Delta_k^2\right)^q 
\end{equation}

From Proposition~\ref{prop:boundmu}, we have that $\bar\Delta_{k} \ge \mu_{D_{k-1}} (x_i)/ \hat c$ for all $x_i\in X_{k-1}$, where $\hat c = \max\{c_1+L^{max}\sqrt{n},c_2\}$.
Since, for all $k \in \{1,\dots,j_{\epsilon}\}$,  a point $x_i\in X_{k-1}$ exists such that $\mu(x_i) > \epsilon$ then 
\[
\frac{\epsilon}{\hat c({\cal C}+1)} \leq \bar\Delta_k
\]
so that, we can write
\[
HI(F(Y^{i+1}(x_k)\cup\{y_k^i+\alpha_k^i d_k^i\})) - HI(F(Y^{i+1}(x_k))) \geq   \left(\gamma c^2\bar\Delta_{k}^2\right)^q \geq 
\left(\gamma \frac{c^2}{\hat c^2({\cal C}+1)^2}\epsilon^2\right)^q
\]
Since, in the expansion procedure, $Y^{i+1}(x_k)$ is updated as $Y^{i+1}(x_k)\cup\{y_k^i+\alpha_k^i d_k^i\}$, we obtain,
summing up the above relation over all function evaluations producing an increase in the hypervolume, 
\[
\overline{HI} - HI_0 \geq Nf_{j_{\epsilon}}^{\mathcal S}\left(\gamma \frac{c^2}{\hat c^2({\cal C}+1)^2}\epsilon^2\right)^q,
\]
that is,
\[
Nf_{j_{\epsilon}}^{\mathcal S} \leq \left\lfloor\frac{(\overline{HI} - HI_0)(\hat c({\cal C}+1))^{2q}}{\gamma^qc^{2q}}\epsilon^{-2q}\right\rfloor.
\]
The desired results hence follows from~\eqref{NF}.
\end{proof}

\section{\DFMOlight\ -- a lighter version of \DFMOnew}\label{sec:DFMOlight}

Algorithm \DFMOnew\ has good theoretical properties in terms of iterations and evaluations complexity. However, these good theoretical properties are counterbalanced by a potentially high computational cost since the method explores all the points in the current list of at every iteration. Drawing inspiration from the method proposed in \cite{custodio2021worst}, we can propose a lighter algorithm, namely \DFMOlight. The new algorithm in fact explores around a single point in the current list at every iteration thereby potentially saving computational resources.

\begin{algorithm}[ht!]
\caption{\DFMOlight}
\begin{algorithmic}[1]
\State {\bf data.} $\theta \in (0,1)$, $c \in (0,1)$, $D\in \Re^{n\times r}$, $L_0=\{(x_j,\tilde \alpha_j),\ x_j \in \Re^n,\ \tilde\alpha_j\in \Re^r_{+},\ j = 1,\dots,|L_0| \}$

\For{$k=0,1,\dots$}

\State select $(x_k,\tilde\alpha_k)\in L_k$ according to some criterion

\State set $X_k=\{x_j: (x_j,\tilde\alpha_j)\in  L_k  \}$

\State set 
$
\Delta_k = \max_{i=1,\dots,r}\{\tilde\alpha_{k}^i\}
$

\State compute $Y_k$ and $\alpha_k$ by the \texttt{Approximate\_optimization} $(x_k,\tilde \alpha_k,X_k,c\Delta_k,D)$ 

\If{$X_k = Y_k$ }
\Statex \phantom{8:} \quad{\em (failure)}

\State set $\tilde\alpha_{k+1}^i = \theta\alpha_k^i$ for all $i=1,\dots,r$

\State set $L_{k+1} \leftarrow (L_k\setminus\{(x_k,\tilde\alpha_k)\})\cup \{(x_k,\tilde\alpha_{k+1})\}$

\Else
\Statex\phantom{8:} \quad{\em (success)}

\State set $\tilde\alpha_{k+1}^i = \alpha_k^i$ for all $i=1,\dots,r$

\State set $\tilde L_k = L_k$

\ForAll {$y\in Y_k$}

\State set $\tilde L_k \leftarrow \tilde L_k\cup \{(y,\tilde\alpha_{k+1})\}$
\EndFor
\State set $L_{k+1}$ non dominated points in  $\tilde L_k$

\EndIf

\EndFor

\end{algorithmic}
\end{algorithm}























Indeed, the main difference between algorithms \DFMOnew\ and \DFMOlight\ consists in how the exploration takes place. \DFMOnew\ performs a complete exploration of the points in $L_k$ at every iteration, whereas \DFMOlight\ only consider one point per iteration. Hence, \DFMOnew\ is able to force theoretical properties on the entire list $L_k$. Instead, in \DFMOlight\ we have to devise a method to keep track of the points that are selected during the iterations. This motivates the introduction of a particular kind of sequences, namely {\em linked sequences}, used to study complexity and convergence of \DFMOlight. 

\subsection{Preliminary results}\ \par

The theoretical properties of Algorithm \DFMOlight\ can be analyzed with reference to particular sequences which are implicitly defined by the algorithm itself. Such sequences are referred to in the literature as {\em linked sequences}. 
In particular, 
drawing inspiration from  \cite{custodio2021worst,liuzzi2016derivative}, we report the formal definition of {\em linked sequence}.
\par\medskip

\begin{definition}\label{def:linked}
Let $\{L_k\}$ with $L_k=\{(x_j,\tilde\alpha_j), \ j=1,\dots, |L_k|\}$ be the
sequence of sets of non-dominated points produced by \DFMOlight. 
\begin{enumerate}
    \item We define a
\emph{linked sequence} 
between iterations $0$ and $k$ ($k>0$) as a (finite) sequence  $\{(x_{j_{\tilde k}},\tilde\alpha_{j_{\tilde k}})\}$ with  ${\tilde k}=1,2,\dots, k$, 
such that, for all $\tilde k$: 
\begin{itemize}
    \item[(a)] either the pair
$(x_{j_{\tilde k}},\tilde\alpha_{j_{\tilde k}})\in L_{\tilde k}$ is generated at iteration ${\tilde k}-1$ of the algorithm by the pair $(x_{j_{{\tilde k}-1}},\tilde\alpha_{j_{{\tilde k}-1}})\in L_{{\tilde k}-1}$;

    \item[(b)]  or it results $L_{\tilde k}\ni (x_{j_{\tilde k}},\tilde\alpha_{j_{\tilde k}}) = (x_{j_{{\tilde k}-1}},\tilde\alpha_{j_{{\tilde k}-1}})\in L_{\tilde k-1}$.
\end{itemize}
 Furthermore, case (a) happens  at least for an index $\tilde k$. 
 \item We denote by $\mathbb{L}_0^k$ the set of linked sequences between iteration 0 and $k$.

 \item   We denote by ${\cal L}_i^{0:k}\in \mathbb{L}_0^k$ the $i$-th linked sequence between iterations 0 and $k$. 

 \item \label{def:K_linked}Given a linked sequence ${\cal L}^{0:k} = \{(x_{j_{\tilde k}},\tilde\alpha_{j_{\tilde k}})\}\in \mathbb{L}_0^k$ between iterations $0$ and $k$, we denote by $\cal K$ the set of iteration indices of the linked sequence such that the pair
$(x_{j_{\tilde k}},\tilde\alpha_{j_{\tilde k}})\in L_{\tilde k}$ is generated at iteration ${\tilde k}-1$ of the algorithm by the pair $(x_{j_{{\tilde k}-1}},\tilde\alpha_{j_{{\tilde k}-1}})\in L_{{\tilde k}-1}$, i.e. case (a) of Definition \ref{def:linked} happens. \end{enumerate}
\end{definition}

\par\smallskip\noindent

It could be noted that algorithm \DFMOlight\ generates many linked sequences.

\par\medskip

For algorithm \DFMOlight, the following proposition states that there is a controlled increase of the hypervolume for every successful iteration.
\par\medskip

\begin{proposition}\label{HIincrease_light}
Let the reference point be $\rho = f_{\max}+s\one$ (with $s>0$ sufficiently large) and ${{\cal L}^{0:k} = }\{(x_{j_{\tilde k}},\tilde\alpha_{j_{\tilde k}})\}\in \mathbb{L}_0^k$ be a linked sequence between iterations $0$ and $k$. Let $\cal K$ be the set of iteration indices introduced in Definition \ref{def:linked} point \ref{def:K_linked}. 
Let  
\[
\Delta_{j_{\tilde k}} = \max_{i=1,\dots,r}\{\tilde\alpha^i_{j_{\tilde k}}\}.
\]

Then, for every $\tilde k\in \cal K$ such that iteration $\tilde k$ is successful, we have
    \[
    HI(F(X_{\tilde k})) - HI(F(X_{{\tilde k-1}})) \geq \left\{\begin{array}{ll}
         \left(\gamma\ c^2\Delta_{j_{\tilde k}}^2\right)^q & \ \text{if}\ \Delta_{j_{\tilde k-1}} = \Delta_{j_{\tilde k}} \\[0.6em]
         \left(\gamma\ \Delta_{j_{\tilde k}}^2\right)^q&  \ \text{if}\ \Delta_{j_{\tilde k-1}} < \Delta_{j_{\tilde k}} 
    \end{array}\right.
    \] 
\end{proposition}

\begin{proof}
    The proof follows the same arguments as in the proof of Proposition \ref{HIincrease}.
\end{proof}

\par\medskip

With reference to a linked sequence between iterations $0$ and $k$, i.e. ${{\cal L}^{0:k} =}\{(x_{j_{\tilde k}},\tilde\alpha_{j_{\tilde k}})\}\in\mathbb{L}_0^k$, let us define the cost function
\[
\Phi_{j_{\tilde k}} = -HI(F(X_{{\tilde k}})) + \eta(\Delta_{j_{\tilde k}}^2)^q.
\]

In the following proposition we show that the cost function sufficiently reduces on the linked sequence ${\cal L}_i^{0:k}$ and for every $k\in {\cal K}_i$.
\par\smallskip

\begin{proposition}\label{phidecrease_linked}
Let ${\cal L}_i^{0:k} = \{(x_{j_{\tilde k}},\tilde\alpha_{j_{\tilde k}})\}\in\mathbb{L}_0^k$ be the $i$-th linked sequence. Then, it holds that,
\[
\Phi_{j_{\tilde k}} - \Phi_{j_{\tilde k-1}} \leq -\tilde c (\Delta_{j_{\tilde k}}^{2})^q\quad\forall\ \tilde k\in {\cal K}_i
\]
with 
\begin{equation}\label{ctilde_light}
\tilde c = \min\left\{\eta\frac{1-\theta^{2q}}{\theta^{2q}}, \gamma^q c^{2q}, \gamma^q-\eta\right\} > 0
\end{equation}
where ${\cal K}_i$ is introduced in Definition \ref{def:linked} point \ref{def:K_linked}.
\end{proposition}
\begin{proof}
For every $\tilde k\in {\cal K}_i$, recall that $(x_{j_{\tilde k-1}},\tilde\alpha_{j_{\tilde k-1}})\in L_{\tilde k-1}$. Then, only three cases can occur, namely: $\Delta_{j_{\tilde k}} = \theta\Delta_{j_{\tilde k-1}}$, $\Delta_{j_{\tilde k}} = \Delta_{j_{\tilde k-1}}$, and $\Delta_{j_{\tilde k}} > \Delta_{j_{\tilde k-1}}$.  
Hence, in the following we consider these three cases.
\begin{itemize}
    \item $\Delta_{j_{\tilde k}} = \theta\Delta_{j_{\tilde k-1}}$, which means that $X_{\tilde k} = X_{\tilde k-1}$. 
    Then, we can write
    \[
    \Phi_{j_{\tilde k}} - \Phi_{j_{\tilde k-1}} \leq \eta(\Delta_{j_{\tilde k}})^{2q} - \eta(\Delta_{j_{\tilde k-1}})^{2q} = \eta\left(1 - \frac{1}{\theta^{2q}}\right)(\Delta_{j_{\tilde k}})^{2q} = -\eta \frac{1 - \theta^{2q}}{\theta^{2q}}(\bar\Delta_{j_{\tilde k}})^{2q}.
    \]
    \item $\Delta_{j_{\tilde k}} = \Delta_{j_{\tilde k-1}}$, which means the iteration was successful, i.e. ${X_{\tilde k} \neq X_{\tilde k-1}}$. Then, by proposition \ref{HIincrease_light}, in this case we can write,
    \[
    \Phi_{j_{\tilde k}} - \Phi_{j_{\tilde k-1}} \leq -(\gamma c^2\Delta_{j_{\tilde k}}^2)^q.
    \]
    \item $\Delta_{j_{\tilde k}} > \Delta_{j_{\tilde k-1}}$, which again means that the iteration was successful, i.e. ${X_{\tilde k} \neq X_{\tilde k-1}}$. Then, again by proposition \ref{HIincrease_light}, we get
    \[
    \Phi_{j_{\tilde k}} - \Phi_{j_{\tilde k-1}} \leq -(\gamma \Delta_{j_{\tilde k}}^2)^q + \eta\Delta_{j_{\tilde k}}^{2q} - \eta\Delta_{j_{\tilde k-1}}^{2q}\leq -(\gamma^q-\eta)\Delta_{j_{\tilde k}}^{2q}.
    \]
\end{itemize}
Thence, the proof is concluded.
\end{proof}

We are now ready to study the complexity results for algorithm \DFMOlight which are the subject of the following subsection.

\subsection{Worst-case complexity of \DFMOlight}\ \par

The first result concerns the definition of a bound on the first iteration $k_\eps$ such $\mu(x_i)\leq \eps$ for at least a point $x_i\in X_{k_\eps}$.

\begin{proposition}\label{prop:bound_iter_light}
For any $\eps >0$, let $k_{\epsilon}$ be the first iteration of algorithm \DFMOlight\ such that   $\mu (x_{i})\leq\epsilon$ for at least a point $x_i\in X_{k_{\epsilon}}$.
Then,
\[
k_{\epsilon} \leq |\mathbb{L}_0^{k_{\epsilon}-1}|\times\left\lfloor\frac{(\overline{HI} + \Phi_0)(c_2({\cal C}+1))^{2q}}{\tilde c}\epsilon^{-2q}\right\rfloor \leq {\cal O}(|{\cal L}(\eps)|\epsilon^{-2q}),
\]
where $\mathbb{L}_0^{k_\epsilon-1}$ is the set of all linked sequences between iterations 0 and $k_\epsilon-1$ and $|{\cal L}(\eps)| = {\cal O}(|\mathbb{L}_0^{k_\epsilon-1}|)$.
\end{proposition}
\begin{proof}
Let ${\cal L}_i=\{(x_{j_k},\tilde\alpha_{j_k})\}\in \mathbb{L}_0^{k_\epsilon-1}$ be the $i$-th  linked sequence between iterations 0 and $k_{\epsilon}-1$. Then, it holds that
\[
\epsilon < \mu(x_{j_k}) = \mu(x_{j_k}) -\mu_{D_{j_k}}(x_{j_k}) + \mu_{D_{j_k}}(x_{j_k}) \leq |\mu(x_{j_k}) -\mu_{D_{j_k}}(x_{j_k})| + \mu_{D_{j_k}}(x_{j_k}) \leq ({\cal C}+1)\mu_{D_{j_k}}(x_{j_k}),
\]
that is $\mu_{D_{j_k}}(x_{j_k})> \epsilon/({\cal C}+1)$ for all $k\in{\cal K}_i$. Then, from proposition \ref{prop:boundmu}, we have for all unsuccessful iterations $j_{k-1}$
\begin{equation}\label{boundbardelta_light}
\frac{\epsilon}{{\cal C}+1} < \mu_{D_{j_{k-1}}}(x_{j_{k-1}}) \leq c_2\Delta_{j_{k}}.
\end{equation}

On the other hand, from proposition \ref{phidecrease_linked}, we have
\[
\Phi_{j_k} - \Phi_{j_{k-1}} \leq -\tilde c(\Delta_{j_k}^2)^q,\qquad \forall\ k\in{\cal K}_i.
\]
Then, summing up the above relations and considering that $\{\Phi_{j_k}\}$ is a nonincreasing sequence, we obtain
\[
-\overline{HI} - \bar\Phi_{0} \leq \Phi_{j_{k_{\epsilon}-1}} - \Phi_{j_0} \leq -\sum_{k\in {\cal K}_i,k\neq 0}\tilde c(\Delta_{j_k}^2)^q,
\]
where $\bar\Phi_0$ is the maximum, among all the linked sequences in $\mathbb{L}_0^{k_{\epsilon}-1}$, of the values $\Phi_{j_0}$.
For every $k\in {\cal K}_i$ let us consider 
\begin{itemize}
    \item ${\cal U}_{{\cal K}_i}\cap\{0,\dots,k-1\}$, i.e. unsuccessful iterations up to the $k-1$-th one;
    \item ${\cal S}_{{\cal K}_i}\cap\{0,\dots,k-1\}$, i.e. successful iterations up to the $k-1$-th one.
\end{itemize}
Then, for every successful iteration $k\in{\cal K}_i$, let $m(k)$ be the biggest index belonging to  ${\cal U}_{{\cal K}_i}\cap\{0,\dots,k-1\}$ so that all the iterations from ${m(k)+1}$ to $k$ are successful. Then, we can write
\[
-\overline{HI} - \bar\Phi_0 \leq \Phi_{j_{k_{\epsilon}-1}} - \Phi_{j_0} \leq -\sum_{k\in {\cal K}_i,k\neq 0}\tilde c(\Delta_{j_k}^2)^q = -\sum_{k\in{\cal U}_{{\cal K}_i}}\tilde c(\Delta_{j_k}^2)^q -\sum_{k\in{\cal S}_{{\cal K}_i}}\tilde c(\Delta_{j_k}^2)^q.
\]
Note that, for every successful iteration $k\in {\cal K}_i$, i.e. $k\in {\cal S}_{{\cal K}_i}$, it results  $\Delta_{j_{k}} \geq \Delta_{j_{m(k)}}$, so that we can write
\[
-\overline{HI} - \bar \Phi_0 \leq \Phi_{j_{k_{\epsilon}-1}} - \Phi_{j_0} \leq -\sum_{k\in{\cal U}_{{\cal K}_i}}\tilde c(\Delta_{j_k}^2)^q -\sum_{k\in{\cal S}_{{\cal K}_i}}\tilde c(\Delta_{j_{m(k)}}^2)^q.
\]
Hence, exploiting \eqref{boundbardelta_light}, we can write
\[
-\overline{HI} - \bar\Phi_0 \leq -|{\cal K}_i|\tilde c \left(\frac{\epsilon}{c_2({\cal C}+1)}\right)^{2q},
\]
that is
\[
\overline{HI} + \bar \Phi_0 \geq |{\cal K}_i|\tilde c \left(\frac{\epsilon}{c_2({\cal C}+1)}\right)^{2q},
\]
\begin{equation}\label{boundutile}
|{\cal K}_i| \leq \left\lfloor\frac{(\overline{HI} + \bar\Phi_0)(c_2({\cal C}+1))^{2q}}{\tilde c}\epsilon^{-2q}\right\rfloor.
\end{equation}
Then, in the worst case, it results
\[
k_{\epsilon} \leq \sum_{{\cal L}_i\in\mathbb{L}_0^{k_{\epsilon}-1}}|{\cal K}_i|.
\]
Thus, recalling the bound in \eqref{boundutile}, we have
\[
k_{\epsilon} \leq |\mathbb{L}_0^{k_{\epsilon}-1}|\times\left\lfloor\frac{(\overline{HI} + \bar \Phi_0)(c_2({\cal C}+1))^{2q}}{\tilde c}\epsilon^{-2q}\right\rfloor,
\]
which concludes the proof.
\end{proof}

\par\smallskip

Let us define, for every set of points $X\subset\Re^n$,
\[
\bar\Gamma(X) = \min_{x\in X}\mu(x).
\]

Then, on the basis of the latter complexity result, we can immediately give the following result concerning asymptotic convergence of \DFMOlight.

\par\smallskip

\begin{corollary}
For Algorithm \DFMOlight, a subset $K$ of iteration indices exists such that
{$\lim_{k\to\infty,k\in K}\bar\Gamma(X_k) = 0$, i.e. $\liminf_{k\to\infty}\bar\Gamma(X_k) = 0$.}
\end{corollary}
\begin{proof}
The proof immediately follows from Proposition \ref{prop:bound_iter_light} and the definition of limit of a subsequence. 
\end{proof}
\par\smallskip

\par\smallskip

The following result pertains to the behavior of \DFMOlight\ on every linked sequence. In particular, considering a linked sequence ${\cal L}_i$, we show that the cardinality of the set of iteration indices $k\in{\cal K}_i$ for which $\mu(x_{j_k})> \epsilon$ is bounded by a constant depending on $\eps^{-2q}$. 
\par\smallskip

\begin{proposition}\label{boundK_light}
Let ${\cal L}_i = \{(x_{j_k},\tilde\alpha_{j_k})\}{\in\mathbb{L}_0^\infty}$ be a{n infinite} linked sequence. For any $\epsilon > 0$, consider the following subset of indices:
	\begin{equation}\label{subset_light}
K_{\eps}^i=\left\{\ k\in {\cal K}_i:\ \mu(x_{j_k}) > \epsilon  \ \right\}.
\end{equation}
 Then,
\[
|K_{\eps}^i| \leq {\cal O}(n^q\epsilon^{-2q}).
\]
\end{proposition}

\begin{proof}
By Assumption \ref{ass:assumption_mu}, 
for all iterations $k = 0,1,\dots$, we have:
\[
\epsilon < \mu(x_{j_{k-1}})  \leq ({\cal C}+1)\mu_{D_{j_{k-1}}}(x_{j_{k-1}}),
\]
Then, from proposition \ref{prop:boundmu}, we have
\begin{equation}\label{boundbardelta1_light}
\frac{\epsilon}{{\cal C}+1} < \mu_{D_{j_{k-1}}}(x_{j_{k-1}}) \leq \hat c\Delta_{j_k}.
\end{equation}
and from proposition \ref{phidecrease_linked}, we have
\begin{equation}\label{complKeps_light}
\Phi_{j_0}- \Phi_{j_k} \geq \sum_{\tilde k\in {\cal K}_i}\tilde c(\Delta_{{j_{\tilde k}}}^2)^q,
\end{equation}
The sequence $\{\Phi_{j_k}\}$ is bounded from below hence  $\Phi^*$ exists such that:
$$\lim_{k\to\infty}\Phi_{j_k}=\Phi^*\ge-\overline{HI}.$$
Then, from (\ref{complKeps_light}), considering that $K_\epsilon^i\subseteq {\cal K}_i$ {(since $\mu(x_i) > \eps$ for all $x_i\in X_0$)}, we obtain:

\[
\bar \Phi_0+\overline{HI} \geq 
\sum_{\tilde k\in {\cal K}_i}\tilde c(\Delta_{{j_{\tilde k}}}^2)^q\geq 
\sum_{ k\in K_{\eps}^i} \tilde c(\Delta_{{j_k}}^2)^q\geq |K_{\epsilon}^i|\tilde c \frac{\epsilon^{2q}}{\hat c^{2q}({\cal C}+1)^{2q}}
\]
Then, we obtain
\[
|K_{\epsilon}^i| \leq \left\lfloor
\frac{(\bar\Phi_0+\overline{HI})\hat c^{2q}({\cal C}+1)^{2q}}{\tilde c}\epsilon^{-2q}
\right\rfloor,
\]
thus concluding the proof.
\end{proof}

\par\smallskip

Concerning the evaluations complexity of \DFMOlight, and with reference to linked sequence, we have the following result.
\par\smallskip

\begin{proposition}\label{complexity_DFMOlight_nf}
For any $\eps >0$, let $j_{\epsilon}$ be the first iteration of algorithm \DFMOlight\ such that  {$\mu(x_i)\leq\eps$ for at least a point $x_i\in X_{j_\eps}$.}
Then, denoting by $Nf_{j_{\epsilon}}$ the number of function evaluations required by Algorithm \DFMOlight\ up to iteration $j_{\epsilon}{-1}$, we have that
$Nf_{j_{\epsilon}} \leq {\cal O}(n|{\cal L}(\epsilon)|\epsilon^{-2q})$. In particular,
\[
Nf_{j_{\epsilon}} \leq {\cal O}(n) |\mathbb{L}_0^{j_{\epsilon}{-1}}|  \left\lfloor\frac{(\overline{HI} + \bar\Phi_0)(c_2({\cal C}+1))^{2q}}{\tilde c}\epsilon^{-2q}\right\rfloor + \left\lfloor\frac{(\overline{HI} - HI_0)(c_2({\cal C}+1))^{2q}}{\gamma^qc^{2q}}\epsilon^{-2q}\right\rfloor,
\]
where $\tilde c$  and $c_2$ are given by~\eqref{ctilde} and in proposition \eqref{prop:boundmu}, respectively.
\end{proposition}

\begin{proof}
First, let us partition the set of iteration indices $\{0,\dots,j_{\epsilon}-1\}$ into $\mathcal S_{j_{\epsilon}}$ and $\mathcal U_{j_{\epsilon}}$ such that
\[
k \in \mathcal S_{j_{\epsilon}} \, \Leftrightarrow \,  L_k \not= L_{k-1}, \quad k\in \mathcal U_{j_{\epsilon}} \, \Leftrightarrow \,  L_k = L_{k-1}, \quad \mathcal S_{j_{\epsilon}} \cup \mathcal U_{j_{\epsilon}} = \{0,\dots,j_{\epsilon}-1\},
\]
that is, $\mathcal S_{j_{\epsilon}}$ and $\mathcal U_{j_{\epsilon}}$ contain the successful and unsuccessful iterations up to $j_{\epsilon}-1$, respectively.



When the algorithm explores a new point, the latter can either succeed to increase the hypervolume or fail to do so. Let us then define $Nf_{j_{\epsilon}}^{\mathcal S}$ as the total number of function evaluations related to points which succeed to increase the hypervolume up to iteration $j_{\epsilon}$. Note that, at each iteration, the maximum number of function evaluations related to points which fail to increase the hypervolume is $r$.
From the instructions of the algorithm, we can write
\begin{equation}\label{NF_light}
Nf_{j_{\epsilon}} \leq r\ {(j_{\epsilon}-1)} + Nf_{j_{\epsilon}}^{\mathcal S} \le {\cal O}(n) |\mathbb{L}_0^{j_{\epsilon}{-1}}| \left\lfloor\frac{(\overline{HI} + \bar\Phi_0)(c_2({\cal C}+1))^{2q}}{\tilde c}\epsilon^{-2q}\right\rfloor + Nf_{j_{\epsilon}}^{\mathcal S},
\end{equation}
where the last inequality follows from Proposition~\ref{prop:bound_iter_light}.

Now, { for any iteration $k < j_{\epsilon}$, let us consider the linked sequence ${\cal L}_i = \{(x_{j_{\tilde k}},\tilde\alpha_{j_{\tilde k}})\}\in \mathbb{L}_0^{j_\eps-1}$ such that $k\in {\cal K}_i$. Let us consider} any index $i \in \{1,\ldots,r\}$
such that the {\tt Approximate\_optimization} procedure succeeds to insert a new point in {$Y^{i+1}$} \st{$\hat L$}. This implies the following relation involving the hypervolume.

\begin{equation}\label{eq:bound_diff_f_light}
HI({F(Y^{i+1}}\cup\{y+\alpha d_{j_k}^i\})) - HI({F(Y^{i+1})}) \geq {(\gamma \alpha^2)^q} \geq \left(\gamma c^2\Delta_{j_k}^2\right)^q. 
\end{equation}
Let us define the index $m(k)$ as follows:
\begin{itemize}
\item if $\mathcal U_{j_{\epsilon}}\cap \{0,\ldots,k-1\}\not= \emptyset$, then $m(k)$ is the largest index of $\mathcal U_{j_{\epsilon}}\cap \{0,\ldots,k-1\}$;
\item otherwise, $m(k) = 0$.
\end{itemize}
The instructions of the algorithm guarantee that ${\Delta}_{j_k} \ge {\Delta}_{j_{m(k)+1}}$.
Hence, from~\eqref{eq:bound_diff_f_light}, it follows that
\[
HI({F(Y^{i+1}}\cup\{y+\alpha d_{j_k}^i\})) - HI({F(Y^{i+1})}) \geq \left(\gamma c^2\Delta_{j_k}^2\right)^q \geq  \left(\gamma c^2\Delta_{j_{m(k)+1}}^2\right)^q 
\]
From Proposition~\ref{prop:boundmu}, we have that {$\mu_{D_{j_{m(k)}}}/c_2\leq \Delta_{j_{m(k)+1}}$.} \st{$\Delta_{m(k)+1} \ge \mu_{D_k} (x_k)/ c_2$.}
Since $\mu_{D_k} (x_i) > \epsilon$ for all $k \in \{0,\dots,j_{\epsilon}-1\}$ {and all $x_i\in X_k$,} we have
\[
\frac{\epsilon}{c_2({\cal C}+1)} \leq \Delta_{j_{m(k)+1}}
\]
so that, we can write
\[
HI({F(Y^{i+1}}\cup\{y+\alpha d_{j_k}^i\})) - HI({F(Y^{i+1})})  \geq   \left(\gamma c^2\Delta_{m(k)+1}^2\right)^q \geq 
\left(\gamma \frac{c^2}{c_2^2({\cal C}+1)^2}\epsilon^2\right)^q
\]

By the instructions of the expansion procedure, we obtain,
summing up the above relation over all function evaluations producing an increase in the hypervolume,  
\[
\overline{HI} - HI_0 \geq Nf_{j_{\epsilon}}^{\mathcal S}\left(\gamma \frac{c^2}{c_2^2({\cal C}+1)^2}\epsilon^2\right)^q,
\]
that is,
\[
Nf_{j_{\epsilon}}^{\mathcal S} \leq \left\lfloor\frac{(\overline{HI} - HI_0)(c_2({\cal C}+1))^{2q}}{\gamma^qc^{2q}}\epsilon^{-2q}\right\rfloor.
\]
The desired results hence follows from~\eqref{NF_light}.
\end{proof}

\section{Special instances of \DFMOlight}\label{sec:special_instances}
In this section we provide an analysis of two special instances of \DFMOlight. In particular, we study the following two algorithms obtained from \DFMOlight\ by selecting the current point according to the following criteria.
\begin{enumerate}
    \item \DFMOmin: at every iteration $k$ the selected pair $(x_k,\tilde\alpha_k)\in\tilde L_k$ is  the one with smallest maximum stepsize.

    \item \DFMOmax: at every iteration $k$ the selected pair $(x_k,\tilde\alpha_k)\in\tilde L_k$ is  the one with highest maximum stepsize.
\end{enumerate}

\subsection{\DFMOmin} In this section we consider the special instance of \DFMOlight\ that, at every iteration $k$, selects (one of) the pair(s) $(x_k,\tilde\alpha_k)\in L_k$ such that its maximum stepsize $\Delta_k$ is the minimum of the maximum stepsizes among all the pairs in the current list $L_k$. In particular, let us define, for every $k$,
\[
\xi_k \stackrel{\triangle}{=} \min_{(x_j,\tilde\alpha_j)\in L_k} \max_{i=1,\dots,r}\tilde\alpha_j^i.
\]
Then, the selected pair $(x_k,\tilde\alpha_k)\in L_k$ is one such that
\[
\xi_k = \Delta_k = \max_{i=1,\dots,r}\{\tilde\alpha_k^i\}.
\]
{Note an important difference between sequences $\{\bar\Delta_k\}$ and $\{\xi_k\}$. The former one is generated by defining $\bar\Delta_{k+1}$ using the final list at iteration $k$, i.e. $\tilde L_k$, whereas the latter one is generated by defining $\xi_k$ using the initial list at iteration $k$, i.e. $L_k$.}

With the above selection strategy, we are able to identify a (very) specific linked sequence, i.e. the one obtained by always analyzing the point with the minimum of the maximum steplengths.

Let us define the improvement function $\Phi$ at iteration $k$ as follows
\[
\Phi_k = -HI(F(X_k)) + \eta\xi_k^{2q}.
\]
Thanks to the definition of algorithm \DFMOmin, in particular the fact that at each iteration  one of the points with minimum of maximum stepsizes is selected,  
it is possible to characterize the behavior of $\Phi_k$ at each iteration and not only in those belonging to the set  $\cal K$.

\begin{proposition}\label{phi_xi_relation}
For every iteration $k=1,\dots$, it results
\[
\Phi_k-\Phi_{k-1} \leq -\tilde c\xi_k^{2q},
\]
where $\tilde c$ is defined in Proposition \ref{phidecrease}.
\end{proposition}
\begin{proof}
For every $k$, only three cases can occur, namely: $\xi_k = \theta\xi_{k-1}$, $\xi_k = \xi_{k-1}$, and $\xi_k > \xi_{k-1}$.  
Hence, in the following we consider these three cases.
\begin{itemize}
    \item $\xi_k = \theta\xi_{k-1}$, which means that $X_{k} = X_{k-1}$. 
    Then, we can write
    \[
    \Phi_{k} - \Phi_{k-1} = \eta\xi_k^{2q} - \eta\xi_{k-1}^{2q} = \eta\left(1 - \frac{1}{\theta^{2q}}\right)\xi_k^{2q} = -\eta \frac{1 - \theta^{2q}}{\theta^{2q}}\xi_k^{2q}.
    \]
    \item $\xi_k = \xi_{k-1}$, which means the iteration was successful, i.e. $X_{k} \neq X_{k-1}$. Then, by proposition \ref{HIincrease_light}, in this case we can write,
    \[
    \Phi_{k} - \Phi_{k-1} \leq -(\gamma c^2\xi_{k}^2)^q.
    \]
    \item $\xi_k > \xi_{k-1}$, which again means that the iteration was successful, i.e. $X_{k} \neq X_{k-1}$. Then, again by proposition \ref{HIincrease_light}, we get
    \[
    \Phi_{k} - \Phi_{k-1} \leq -(\gamma \xi_k^2)^q + \eta\xi_k^{2q} - \eta\xi_{k-1}^{2q}\leq -(\gamma^q-\eta)\xi_k^{2q}.
    \]
\end{itemize}
Hence, the proof is concluded.
\end{proof}
\par\smallskip

\begin{proposition}\label{boundK_min}
For any $\epsilon > 0$, assume that $\mu(x_i)>\eps$ for all $x_i\in X_0$. Let $j_\epsilon$ be the first iteration such that 
\[
\mu(x_{j_\epsilon}) \leq \epsilon
\]
Then,
\[
j_{\eps} \leq {\cal O}(\epsilon^{-2q}).
\]
\end{proposition}

\begin{proof}
Since $j_\eps$ is the first iteration such that $\mu(x_{j_\eps}) \leq\eps$, then we have that $\mu(x_k) > \eps$ for all $k=0,\dots,j_\eps -1$.
By Assumption \ref{ass:assumption_mu} {and recalling that $\mu(x_i)>\eps$ for all $x_i\in X_0$}, 
for all iterations $k = 1,\dots$, we have:
\[
\epsilon < \mu(x_{k-1})  \leq ({\cal C}+1)\mu_{D_{k-1}}(x_{{k-1}}).
\]
Let us now consider an iteration $k<j_\eps$, then, from proposition \ref{prop:boundmu}, we can write
\begin{equation}\label{boundbardelta1_min}
\frac{\epsilon}{{\cal C}+1} < \mu_{D_{{h}}}(x_{{h}}) {\leq \hat c\xi_{k}}
\end{equation}
for some iteration $h< k$.

Now, from proposition \ref{phidecrease_linked} and recalling that $\Phi^*$ exists such that $\Phi_k\geq\Phi^*$ and that $\Phi^*\geq-\overline{HI}$, we have
\begin{equation}\label{complKeps_min}
\Phi_0+\overline{HI}\geq \Phi_0 - \Phi^*\geq \Phi_{0}- \Phi_{j_{\eps}} \geq \sum_{k=1}^{j_{\eps}}\tilde c(\xi_{k}^2)^q,
\end{equation}
Then, from (\ref{boundbardelta1_min}) and \eqref{complKeps_min}, we obtain:

\[
\Phi_0+\overline{HI} \geq 
\sum_{k=1}^{j_{\eps}}\tilde c(\xi_{k}^2)^q\geq 
j_{\eps}\,\tilde c \frac{\epsilon^{2q}}{\hat c^{2q}({\cal C}+1)^{2q}}
\]
Then, we obtain
\[
j_\eps \leq \left\lfloor\frac{(\Phi_0+\overline{HI})(\hat c({\cal C}+1))^{2q}}{\tilde c}\eps^{-2q}\right\rfloor 
\]
concluding the proof.
\end{proof}

\par\smallskip

\begin{corollary}
Let $\{x_k\}$ be the sequence of points $x_k\in X_k$ selected at each iteration by algorithm \DFMOmin. Then, a subset of iteration indices $K$ exists such that 
$\lim_{k\to\infty,k\in K}\mu(x_k) = 0$, i.e. $\liminf_{k\to\infty}\mu(x_k) = 0$. 
\end{corollary}
\begin{proof}
    The proof immediately follows from the definition of liminf and Proposition \ref{boundK_min}.
\end{proof}
\par\medskip

\begin{proposition}
For any $\eps >0$, assume that $\mu(x_i)>\eps$ for all $x_i\in X_0$. Let $j_{\epsilon}$ be the first iteration of algorithm \DFMOmin\ such that  $\mu(x_{j_{\epsilon}})\leq\epsilon$. Then, denoting by $Nf_{j_{\epsilon}}$ the number of function evaluations required by Algorithm \DFMOmin\ up to iteration $j_{\epsilon}$, we have that
$Nf_{j_{\epsilon}} \leq {\cal O}(n\epsilon^{-2q})$. In particular,
\[
Nf_{j_{\epsilon}} \leq {\cal O}(n)  \left\lfloor\frac{(\overline{HI} + \Phi_0)(\hat c({\cal C}+1))^{2q}}{\tilde c}\epsilon^{-2q}\right\rfloor + \left\lfloor\frac{(\overline{HI} - HI_0)\hat c^{2q}({\cal C}+1)^{2q}}{\gamma^qc^{2q}}\epsilon^{-2q}\right\rfloor,
\]
where $\hat c = \max\{c_1+L^{max}\sqrt{n},c_2\}$ and $\tilde c$, and $c_1,c_2$ are given by~\eqref{ctilde} and in proposition \eqref{prop:boundmu}, respectively.
\end{proposition}

\begin{proof}
First, let us partition the set of iteration indices $\{0,\dots,j_{\epsilon}-1\}$ into $\mathcal S_{j_{\epsilon}}$ and $\mathcal U_{j_{\epsilon}}$ such that
\[
k \in \mathcal S_{j_{\epsilon}} \, \Leftrightarrow \,  L_k \not= L_{k-1}, \quad k\in \mathcal U_{j_{\epsilon}} \, \Leftrightarrow \,  L_k = L_{k-1}, \quad \mathcal S_{j_{\epsilon}} \cup \mathcal U_{j_{\epsilon}} = \{0,\dots,j_{\epsilon}-1\},
\]
that is, $\mathcal S_{j_{\epsilon}}$ and $\mathcal U_{j_{\epsilon}}$ contain the successful and unsuccessful iterations up to $j_{\epsilon}-1$, respectively.



When the algorithm explores a new point, the latter can either succeed to increase the hypervolume or fail to do so. Let us then define $Nf_{j_{\epsilon}}^{\mathcal S}$ as the total number of function evaluations related to points which succeed to increase the hypervolume up to iteration $j_{\epsilon}$. Note that, at each iteration, the maximum number of function evaluations related to points which fail to increase the hypervolume is $r$.
From the instructions of the algorithm, we can write
\begin{equation}\label{NF_min}
Nf_{j_{\epsilon}} \leq r\ j_{\epsilon} + Nf_{j_{\epsilon}}^{\mathcal S} \le {\cal O}(n)  \left\lfloor\frac{(\overline{HI} + \Phi_0)(\hat c({\cal C}+1))^{2q}}{\tilde c}\epsilon^{-2q}\right\rfloor + Nf_{j_{\epsilon}}^{\mathcal S},
\end{equation}
where the last inequality follows from Proposition~\ref{prop:bound_iter_light}.
Now, let us consider any iteration $k < j_{\epsilon}$ and any index $i \in \{1,\ldots,n\}$
such that the line search succeeds to insert a new point in $\hat L$. This implies the following relation involving the hypervolume.

\begin{equation}\label{eq:bound_diff_f_min}
HI(\hat L\cup\{y+\alpha d_k^i\}) - HI(\hat L) \geq (\gamma \beta^2)^q \geq \left(\gamma c^2\Delta_k^2\right)^q 
\end{equation}

From Proposition~\ref{prop:boundmu}, we have that $\Delta_{m(k)+1} \ge \mu_{D_k} (x_k)/ \hat c$ with $\hat c = \max\{c_1+L^{max}\sqrt{n},c_2\}\leq {\cal O}(n^{1/2})$.
Since $({\cal C}+1)\mu_{D_k} (x_k) \geq \mu(x_k)> \epsilon$ for all $k \in \{0,\dots,j_{\epsilon}-1\}$, we have
\[
\frac{\epsilon}{\hat c({\cal C}+1)} \leq \Delta_k
\]
so that, from \eqref{eq:bound_diff_f_min}, we can write
\[
HI(\hat L\cup\{y+\alpha d_k^i\}) - HI(\hat L) \geq  
\left(\gamma \frac{c^2}{\hat c^2({\cal C}+1)^2}\epsilon^2\right)^q
\]
Since, in the expansion procedure, the new $\hat L$ is set equal to $\hat L\cup\{y+\alpha d_k^i\}$, we obtain,
summing up the above relation over all function evaluations producing an increase in the hypervolume, 
\[
\overline{HI} - HI_0 \geq Nf_{j_{\epsilon}}^{\mathcal S}\left(\gamma \frac{c^2}{\hat c^2({\cal C}+1)^2}\epsilon^2\right)^q,
\]
that is,
\[
Nf_{j_{\epsilon}}^{\mathcal S} \leq \left\lfloor\frac{(\overline{HI} - HI_0)\hat c^{2q}({\cal C}+1)^{2q}}{\gamma^qc^{2q}}\epsilon^{-2q}\right\rfloor.
\]
The desired results hence follows from~\eqref{NF_min}.
\end{proof}

\subsection{\DFMOmax}
In this section we consider a special version of \DFMOlight, namely \DFMOmax. In particular, 
\DFMOmax, at every iteration, selects the pair $(x_k,\tilde\alpha_k)\in L_k$ such that 
\[
\bar\Delta_k = \Delta_k = \max_{i=1,\dots,r}\{\tilde\alpha_k^i\} 
\]
where 
\[
\bar\Delta_k = \max_{(x_j,\tilde\alpha_j)\in L_k} \max_{i=1,\dots,r}\tilde\alpha_j^i.
\]
Hence, \DFMOmax\ selects (one of) the pair(s) such that its maximum stepsize $\Delta_k$, is the maximum of the maximum stepsizes among all the pairs in the current list $L_k$, i.e. $\bar\Delta_k$.

The advantage of \DFMOmax\ versus \DFMOlight\ (considered in section \ref{sec:DFMOlight}) consists in the stationarity results they are able to achieve. For \DFMOlight\ we are able to prove that within a maximum number of iterations the algorithm produces a set $X_k$ such that at least a point in it has a stationarity measure below a given tolerance. For \DFMOmax, at the cost of computing $\bar \Delta_k$ at each iteration, we are able to prove that within a maximum number of iterations the algorithm produces a set $X_k$ such that all the points in it have a stationarity below a prefixed tolerance. Hence, a somewhat stronger result. 

To carry out the complexity analysis of this instance of the algorithm \DFMOlight, we need a slightly modified version of propositions \ref{phidecrease} and \ref{phidecrease_linked}. This is because, in this version of the algorithm we shall bound the number of iterations such that $\bar\Delta_k$ falls below a pre-specified tolerance. Then, we can connect the stepsizes at iteration $k$, $\Delta_k$,  with stationarity at the previous iteration, $\mu(x_i)$ with $x_i\in X_{k-1}$.

\begin{proposition}\label{phidecrease_special}
Let ${\cal L}_i^{0:k} = \{(x_{j_{\tilde k}},\tilde\alpha_{j_{\tilde k}})\}\in\mathbb{L}_0^k$ be the $i$-th linked sequence. Then, it holds that,
\[
\Phi_{j_{\tilde k}} - \Phi_{j_{\tilde k-1}} \leq -
{\check c} (\bar \Delta_{j_{\tilde k-1}}^{2})^q\quad\forall\ \tilde k\in {\cal K}_i
\]
with 
\begin{equation}\label{ctilde_special}
{\check c} = \min\left\{\eta({1-\theta^{2q}}), \gamma^q c^{2q}, \gamma^q-\eta\right\} > 0
\end{equation}
where ${\cal K}_i$ is defined in Definition \ref{def:linked} point \ref{def:K_linked}.
\end{proposition}
\begin{proof}
For every $\tilde k\in {\cal K}_i$, recall that $(x_{j_{\tilde k-1}},\tilde\alpha_{j_{\tilde k-1}})\in L_{\tilde k-1}$. Then, only three cases can occur, namely: $\Delta_{j_{\tilde k}} = \theta\bar \Delta_{j_{\tilde k-1}}$, $\Delta_{j_{\tilde k}} = \bar\Delta_{j_{\tilde k-1}}$, and $\Delta_{j_{\tilde k}} > \bar\Delta_{j_{\tilde k-1}}$.  
Hence, in the following we consider these three cases.
\begin{itemize}
    \item $\Delta_{j_{\tilde k}} = \theta\bar\Delta_{j_{\tilde k-1}}$, which means that $X_{\tilde k} = X_{\tilde k-1}$. 
    Then, we can write
    \[
    \Phi_{j_{\tilde k}} - \Phi_{j_{\tilde k-1}} \leq \eta(\Delta_{j_{\tilde k}})^{2q} - \eta(\bar \Delta_{j_{\tilde k-1}})^{2q} = -\eta ({1 - \theta^{2q}})(\bar\Delta_{j_{\tilde k-1}})^{2q}.
    \]
    \item $\Delta_{j_{\tilde k}} = \bar\Delta_{j_{\tilde k-1}}$, which means the iteration was successful, i.e. $X_{\tilde k} \neq X_{\tilde k-1}$. Then, by proposition \ref{HIincrease_light}, in this case we can write,
    \[
    \Phi_{j_{\tilde k}} - \Phi_{j_{\tilde k-1}} \leq -(\gamma c^2\bar\Delta_{j_{\tilde k-1}}^2)^q.
    \]
    \item $\Delta_{j_{\tilde k}} > \bar\Delta_{j_{\tilde k-1}}$, which again means that the iteration was successful, i.e. $X_{\tilde k} \neq X_{\tilde k-1}$. Then, again by proposition \ref{HIincrease_light}, we get
    \[
    \Phi_{j_{\tilde k}} - \Phi_{j_{\tilde k-1}} \leq -(\gamma \bar\Delta_{j_{\tilde k}}^2)^q + \eta\bar\Delta_{j_{\tilde k}}^{2q} - \eta\bar\Delta_{j_{\tilde k-1}}^{2q}\leq -(\gamma^q-\eta)\bar \Delta_{j_{\tilde k}}^{2q} < -(\gamma^q-\eta)\bar\Delta_{j_{\tilde k-1}}^{2q}.
    \]
\end{itemize}
Thence, the proof is concluded.
\end{proof}

In the next proposition we bound the number of iterations such that $\bar\Delta_k > \delta$ for a pre-specified threshold $\delta>0$.

\begin{proposition}\label{prop:bound_iter_special}
For any $\delta >0$, let $k_{\delta}$ be the first iteration of algorithm \DFMOmax\ such that   $\bar\Delta_k \leq \delta$.
Then,
\[
k_{\delta} \leq |\mathbb{L}_0^{k_{\delta}}|\times\left\lfloor\frac{(\overline{HI} + \bar\Phi_0)}{\tilde c}\delta^{-2q}\right\rfloor \leq {\cal O}(|\mathbb{L}_0^{k_{\delta}}|\delta^{-2q}),
\]
where $\mathbb{L}_0^{k_\delta}$ is the set of all linked sequences between iterations 0 and $k_\delta$.
\end{proposition}

\begin{proof}
Since $k_\delta$ is the first iteration such that $\bar\Delta_{k_\delta} \leq \delta$, then $\bar\Delta_k > \delta$ for all $k=0,\dots,k_{\delta}-1$.

On the other hand, from proposition \ref{phidecrease_special}, we have
\[
\Phi_{j_k} - \Phi_{j_{k-1}} \leq -\check c(\bar\Delta_{j_{k-1}}^2)^q,\qquad \forall\ k\in{\cal K}_i.
\]
Then, summing up the above relations and considering the nonincreasing feature of $\{\Phi_{l_k}\}$, we obtain
\[
-\overline{HI} - \bar\Phi_{0} \leq \Phi_{j_{k_{\delta}}} - \Phi_{j_0} \leq -\sum_{k\in {\cal K}_i,k\neq 0}\check c(\bar\Delta_{j_{k-1}}^2)^q\leq -\sum_{k\in {\cal K}_i,k\neq 0}\check c\delta^{2q} =-|{\cal K}_i|\check c \delta^{2q},
\]
where $\bar\Phi_0$ is the maximum, among all the linked sequences in $\mathbb{L}_0^{k_{\epsilon}}$, of the values $\Phi_{j_0}$.
Then, we can write
\[
\overline{HI} + \bar\Phi_0 \geq |{\cal K}_i|\check c\delta^{2q},
\]
\[
|{\cal K}_j| \leq \left\lfloor\frac{(\overline{HI} + \bar\Phi_0)}{\check c}\delta^{-2q}\right\rfloor.
\]
Now, it results
\[
k_{\delta} \leq \sum_{{\cal L}_i\in\mathbb{L}_0^{k_{\delta}}}|{\cal K}_i|\leq |\mathbb{L}_0^{k_{\delta}}|\times\left\lfloor\frac{(\overline{HI} + \bar\Phi_0)}{\check c}\delta^{-2q}\right\rfloor,
\]
which concludes the proof.
\end{proof}

\par\medskip

Finally, we can give the complexity result for \DFMOmax.

\begin{proposition}\label{complexity_DFMOstrong}
{Let $\Delta_i = \max_{j=1,\dots,r}\{\tilde\alpha_i^j\} > \eps$ for all pairs $(x_i,\tilde\alpha_i)\in \tilde L_0$.} For any $\eps>0$, let $k_\eps$ be the first iteration such that $\Gamma(X_{k_\eps-1})\leq\eps$.
Then, 
\[
k_\eps \leq |\mathbb{L}_0^{k_{\eps}}|\times\left\lfloor\frac{(\overline{HI} + \bar\Phi_0)(\hat c({\cal C}+1))^{2q}}{\check c}\eps^{-2q}\right\rfloor.
\]
\end{proposition}
\begin{proof}
Let us define 
\[
\delta = \dfrac{\eps}{({\cal C}+1)\hat c}.
\]
Hence, by proposition \ref{prop:bound_iter_special}, the algorithm takes $k_\eps$ iterations with  
\[
k_\eps \leq |\mathbb{L}_0^{k_{\eps}}|\times\left\lfloor\frac{(\overline{HI} + \bar\Phi_0)(\hat c({\cal C}+1))^{2q}}{\check c}\eps^{-2q}\right\rfloor
\]
to obtain $\bar\Delta_{k_\eps} \leq \eps/(c_2({\cal C}+1))$. To prove that $\Gamma(X_{k_\eps -1})\leq\eps$, we show that $\mu(x_i)\leq\eps$ for all $x_i\in X_{k_\eps-1}$. In fact, {since $\Delta_i = \max_{j=1,\dots,r}\{\tilde\alpha_i^j\} > \eps$ for all pairs $(x_i,\tilde\alpha_i)\in \tilde L_0$,} each point $x_i\in X_{k_\eps-1}$ belongs to a linked sequence ${\cal L}_i=\{(x_{j_k},\tilde\alpha_{j_k})\}$ and  $x_i = x_{j_{\tilde k-1}}$ with $\tilde k-1$ the largest index in ${\cal K}_i$. 
By proposition \ref{prop:boundmu} we have 
\[
\mu(x_{j_{\tilde k-1}})\leq ({\cal C}+1)\hat c\Delta_{j_{\tilde k}} = ({\cal C}+1) \hat c\bar\Delta_{j_{k_\eps}} \leq 
\eps. 
\]
So the proof is concluded.
\end{proof}

\par\medskip

\begin{corollary}
For Algorithm \DFMOmax, a subset $K$ of iteration indices exists such that:
\begin{itemize}
    \item[(i)] $\lim_{k\to\infty,k\in K}\Gamma(X_k) = 0$ {(where $\Gamma(X_k)$ is  given by (\ref{Gamma})}, i.e. $\liminf_{k\to\infty}\Gamma(X_k) = 0$;

    \item[(ii)] if $\{x_k\}$ is a sequence such that $x_k\in X_k$ for all $k$, $\liminf_{k\to\infty}\mu(x_k) = 0.$ \end{itemize}

\end{corollary}
\begin{proof}
Point (i) immediately follows from Proposition \ref{complexity_DFMOstrong} and the definition of limit of a subsequence. 

Now, let us consider point (ii). Let $\{x_k\}$ be e sequence such that $x_k\in X_k$. By the definition of $\Gamma$, we can write that
\[
0\leq \mu(x_k)\leq \Gamma(X_k),
\]
so that the result follows by point (i) above.
\end{proof}

\par\smallskip

\begin{proposition}
For any $\eps >0$, let $k_{\epsilon}$ be the first iteration of algorithm \DFMOmax\ such that  $\Gamma(X_{k_{\epsilon}-1})\leq\epsilon$.Then, denoting by $Nf_{k_{\epsilon}}$ the number of function evaluations required by Algorithm \DFMOmax\ up to iteration $k_{\epsilon}$, we have that
$Nf_{k_{\epsilon}} \leq {\cal O}(n^{q+1}|{\cal L}(\epsilon)|\epsilon^{-2q})$. In particular,
\[
Nf_{k_{\epsilon}} \leq {\cal O}(n) |\mathbb{L}_0^{j_{\epsilon}}| \left\lfloor\frac{(\overline{HI} + \bar\Phi_0)(\hat c({\cal C}+1))^{2q}}{\check c}\epsilon^{-2q}\right\rfloor + \left\lfloor\frac{(\overline{HI} - HI_0)\hat c^{2q}({\cal C}+1)^{2q}}{\gamma^qc^{2q}}\epsilon^{-2q}\right\rfloor,
\]
where $\tilde c$  and $c_2$ are given by~\eqref{ctilde} and in proposition \eqref{prop:boundmu}, respectively.
\end{proposition}
\begin{proof}
The proof follows exactly the same arguments as the proof of Theorem \ref{complexity_DFMOlight_nf}.     
\end{proof}






\section{Conclusions}\label{sec:conc}

In the following table we collect the comcplexity results of the proposed algorithms for better clarity.

\begin{table}[ht!]
\begin{center}
\renewcommand{\arraystretch}{1.2}
\begin{tabular}{l|c|c|c}\hline
Algorithm       &  $|K_{\epsilon}|$   & $j_\eps$ & $Nf_\eps$\\\hline
     \DFMOnew   &  $\left\lfloor
\frac{(\Phi_0+\overline{HI})(\hat c({\cal C}+1))^{2q}}{\tilde c}\epsilon^{-2q}
\right\rfloor$   & $\left\lfloor
\frac{(\Phi_0+\overline{HI})(\hat c({\cal C}+1))^{2q}}{\tilde c}\epsilon^{-2q}
\right\rfloor$ & ${\cal O}(n^{q+1}L(\eps)\eps^{-2q})$\\
     \DFMOlight &  -- & $|\mathbb{L}_0^{j_{\epsilon}-1}|\left\lfloor\frac{(\bar \Phi_0+\overline{HI})(c_2({\cal C}+1))^{2q}}{\tilde c}\epsilon^{-2q}\right\rfloor$ & ${\cal O}(n|{\cal L}(\epsilon)|\epsilon^{-2q})$\\
     \DFMOmin   & -- & $\left\lfloor\frac{(\Phi_0+\overline{HI})(\hat c({\cal C}+1))^{2q}}{\tilde c}\eps^{-2q}\right\rfloor $ & ${\cal O}(n^{q+1}\eps^{-2q})$\\
     \DFMOmax   & -- & $|\mathbb{L}_0^{j_{\eps}}|\left\lfloor\frac{(\bar\Phi_0+\overline{HI})(\hat c({\cal C}+1))^{2q}}{\check c}\eps^{-2q}\right\rfloor$ & ${\cal O}(n^{q+1}|{\cal L}(\eps)|\eps^{-2q})$\\\hline
\end{tabular}    
\renewcommand{\arraystretch}{1}
\end{center}
    \caption{Worst-case complexity bounds}
    \label{tab:bounds}
\end{table}
\par\smallskip

Since all of the algorithm presented in the paper deal with set of points (i.e. the set of nondominated points $X_k$ at each iteration $k$), it is necessary to specify what $K_\eps$ and $j_\eps$ refer to in the various algorithms. Thence, we recall the following definitions.
\begin{itemize}
    \item[-] For algorithm \DFMOnew, $K_\eps$ is defined in \eqref{subset}, i.e.
    \[
    K_\eps = \{k\in\{0,1,\dots\}:\ \Gamma(X_k) > \eps\}
    \]
    that is the set of iteration indices such that $\mu(x_k) > \eps$ for at least a point $x_k\in X_k$.

    Furthermore, $j_\eps$ is the index of the first iteration such that $\mu(x_{i}) \leq \eps$ for all $x_{i}\in X_{j_\eps}$.

    \item[-] For algorithm \DFMOlight, $j_\eps$ is the index of the first iteration such that $\mu(x_i)\leq\eps$ for at least a point $x_i\in X_{j_\eps}$.

    \item[-] For algorithm \DFMOmin, $j_\eps$ is the index of the first iteration such that $\mu(x_{j_\eps})\leq\eps$, where $x_{j_\eps}\in X_{j_\eps}$.

    \item[-] For algorithm \DFMOmax, $j_\eps$ is the index of the first iteration such that $\Gamma(X_{j_\eps -1}) \leq \eps$, i.e. $\mu(x_i)\leq \eps$ for all $x_i\in X_{j_\eps-1}$.
\end{itemize}

For the sake of better readability, in Table \ref{tab:constants} we report the constants defining bounds given in Table \ref{tab:bounds}.

\begin{table}[ht!]
\begin{center}
\renewcommand{\arraystretch}{1.2}
\begin{tabular}{l|c|c}\hline
constant & definition & where is defined \\\hline
  $\overline{HI}$   & max. hypervolume & -- \\
  $c_1$   & $\left(\displaystyle\frac{2(L^{max} + \gamma)}{\delta(1-\delta)}   + L^{max}\sqrt{n}\right)$ & Proposition \ref{prop:boundmu}\\
  $c_2$   & $\left(\displaystyle\frac{L^{max} + 2\gamma}{2\theta} \right)$ & Proposition \ref{prop:boundmu}\\
  $\tilde c$ & $\min\left\{\eta\frac{1-\theta^{2q}}{\theta^{2q}}, \gamma^q c^{2q}, \gamma^q-\eta\right\}$&  Proposition \ref{phidecrease}\\
  $\Phi_0$   & -- &  Proposition \ref{prop:Kstrong}\\
  $\hat c$ & $\max\{c_1+L^{max}\sqrt{n},c_2\}$ & Proposition \ref{prop:Kstrong} \\
  $L(\eps)$ & $\max_{k=0,\dots,j_\eps}|L_k|$ & Proposition \ref{prop:NF_DFMOnew}\\
  $\mathbb{L}_0^{k}$ & set of all linked seq. between iterations $0$ and $k$ & Proposition \ref{prop:bound_iter_light} \\
  $|{\cal L}(\eps)|$ & ${\cal O}(|\mathbb{L}_0^{j_\eps}|)$ & Proposition \ref{prop:bound_iter_light}\\
  $\check c$ & $\min\left\{\eta({1-\theta^{2q}}), \gamma^q c^{2q}, \gamma^q-\eta\right\}$ & Proposition \ref{phidecrease_special}\\
  \hline
\end{tabular}    
\renewcommand{\arraystretch}{1}
\end{center}
    \caption{Constant definitions}
    \label{tab:constants}
\end{table}


\section*{Disclosure of interests}
The authors have no competing interests to declare that are relevant to the content of this article.

\section*{Data availability statement}
{The authors declare that the data supporting the findings of this study are available within the paper.
}

\appendix

\section{Appendix: A technical result}

In this section we report an interesting characterization connecting hypervolume increase with the existence of a new non-dominated point.
\par\smallskip

\begin{proposition}\label{HIincrease_gen_new}
Let $Y\subseteq\Re^n$ be a set of points and let the reference point be $\rho = f_{\max}+s\mathbf{1}$ (where $\mathbf{1}$ is the vector of all ones of appropriate dimension and $s>0$ a  sufficiently large scalar parameter). Then, $y\in\Re^n$ is such that 
\begin{equation}\label{eq:nondomnew}
F(x) \not\leq F(y),\ F(y)\neq F(x) 
\qquad\forall\ x\in Y
\end{equation}
if and only if
\[
HI(F(Y\cup\{y\})) - HI(F(Y)) >0. 
\]
\end{proposition}
\begin{proof} (Only if part). From \eqref{eq:nondomnew}, for all $x\in Y$ an index $\ell\in\{1,\dots,q\}$ exists such that
    \[
    f_\ell(y)< f_\ell(x) .
    \]
    Let us define the following subsets $Y_j\subseteq Y$, for $j=1,\ldots,q$:
\[
Y_j=\{ x\in Y:f_j(y) <f_j(x) \}
\]
 and the following scalar $\varepsilon>0$:
\[\varepsilon\ =\min_{j=1,\ldots,q} \quad \min_{\ x\in Y_j} f_j(x) -f_j(y). 
\]
Then, 
\[
F(y) \not >  F(x) -\eps\mathbf{1} \qquad\forall\ x\in Y
\]
 Now, we can consider the following set:
    \[
    B = \left\{w\in\Re^q:\ f_i(y) \leq w_i \leq f_i(y)+\varepsilon, \ i=1,\dots,q\right\},
    \]
    such that $Vol(B) = \varepsilon^q$.
    Then, it holds that
    \[
    int(B)\cap \left(\bigcup_{a\in F(Y)}[a,\rho]\right) = \emptyset. 
    \]
    Indeed, if the above intersection was not empty, a point $w\in\Re^q$ would exist such that $w_i< f_i(y)+\nu$ for all $i=1,\dots,q$. Furthermore, either $x_w$ exists such that $w = f(x_w)\in F(Y)$ or $w$ is dominated by a point in $F(Y)$. 
    \begin{enumerate}
        \item if $w\in F(Y)$, we know that an index $\ell$ exists such that $f_\ell(y) \leq f_\ell(x_w)-\varepsilon = w_\ell-\varepsilon$ which contradicts  $w_i< f_i(y)+\varepsilon$ for all $i=1,\dots,q$.

        \item if $w$ is dominated by a point in $F(Y)$, there will be a point $\tilde w\in F(Y)$ dominating $w$. Then, we again obtain a contradiction by reasoning as in the previous point.
    \end{enumerate}

    Now, we show that 
    \[
    B\subseteq \left(\bigcup_{a\in F(Y\cup\{y\})}[a,\rho]\right). 
    \]
    To this aim,  let us consider
    \[
    C = B \cap \{w\in\Re^q:\ w_i \leq \rho_i,\ i=1,\dots,q\}.
    \]
    The previous inclusion does not hold only if  $C\neq B$ but this imply $f_i(y) + \varepsilon > \rho_i$ for some index $i=1,\dots,q$. But this cannot be since $s> 0$ and sufficiently large.
    Finally, recalling that
    \[
    HI(F(Y)) = \vol\left(\bigcup_{a\in F(Y)}[a,\rho]\right).
    \]
    we have
    \[
    HI(F(Y\cup\{y\})) \geq  HI(F(Y)) + Vol(B) = HI(F(Y)) + \varepsilon^q,
    \]    
    which concludes the proof of the "only if" part.
    \par\medskip\noindent
    (If). For simplicity we introduce the following sets:
    \begin{eqnarray*}
        && M=\bigcup_{a\in F(Y)}[a,\rho] \\
        &&N=[F(y),\rho]
    \end{eqnarray*}
    By assumption we have:
    \begin{eqnarray*}
 &&\vol(M\cup N)-\vol(M)>0
    \end{eqnarray*}
    from which:
        \begin{eqnarray*}
 &&\vol(M\cup N)-\vol(M)=\vol(M)+\vol(N\setminus (N\cap M))-\vol(M) >0\\
 &&\vol(N\setminus (N\cap M))>0.
 \end{eqnarray*}
 This implies that a point $\bar w\in N$ and a scalar $\varepsilon>0$ exist such the set
  \[
    \bar B = \left\{w\in\Re^q:\ |\bar w_i - w_i| \leq \varepsilon, \ i=1,\dots,q\right\},
    \]
    satisfies:
    \[
    \bar B \subset N\setminus (N\cap M) .
    \]
    Since $ \bar B \subset N$ implies that
    $$\bar w \in int(N)=int([F(y),\rho]) $$
    and, hence, we have
\begin{equation}\label{rel1}
 F(y)< \bar w
 \end{equation}
Then $$\bar B\cap M=\bar B\cap\Big( \bigcup_{a\in F(Y)}[a,\rho]\Big)=\emptyset$$ ensures that for every $x\in F(Y)$  an index $\ell\in\{1,\dots,q\}$ must exist such that
$$\bar w_\ell+\varepsilon <f_\ell(x)$$
and recalling (\ref{rel1}) that
$$f_\ell(y)+\varepsilon <f_\ell(x)$$
which proves (\ref{eq:nondomnew}).
\end{proof}

Note that the reference vector $\rho$ (in the objective space) used in Proposition \ref{HIincrease_gen} is required to be $\rho = f_{\max} + s\mathbf{1}$ with $s$ a sufficiently large positive number. To be able to prove the result of Proposition \ref{HIincrease_gen}, it is not possible to use simply the vector $f_{\max}$. In the following remark, we evidence that a pathological situation might arise if we consider $\rho=f_{\max}$, namely the sufficient increase in the hypervolume indicator might fail to happen.

\begin{remark}\label{remarkA.1}
Let us consider the following bi-objective optimization problem
\[
\min\ x^2,\ \dfrac{(x-4)^2}{18}.
\]
In figure \ref{fig_g1} we report the graphs of the objective functions.
\begin{figure}
    \centering    \includegraphics[width=0.7\textwidth]{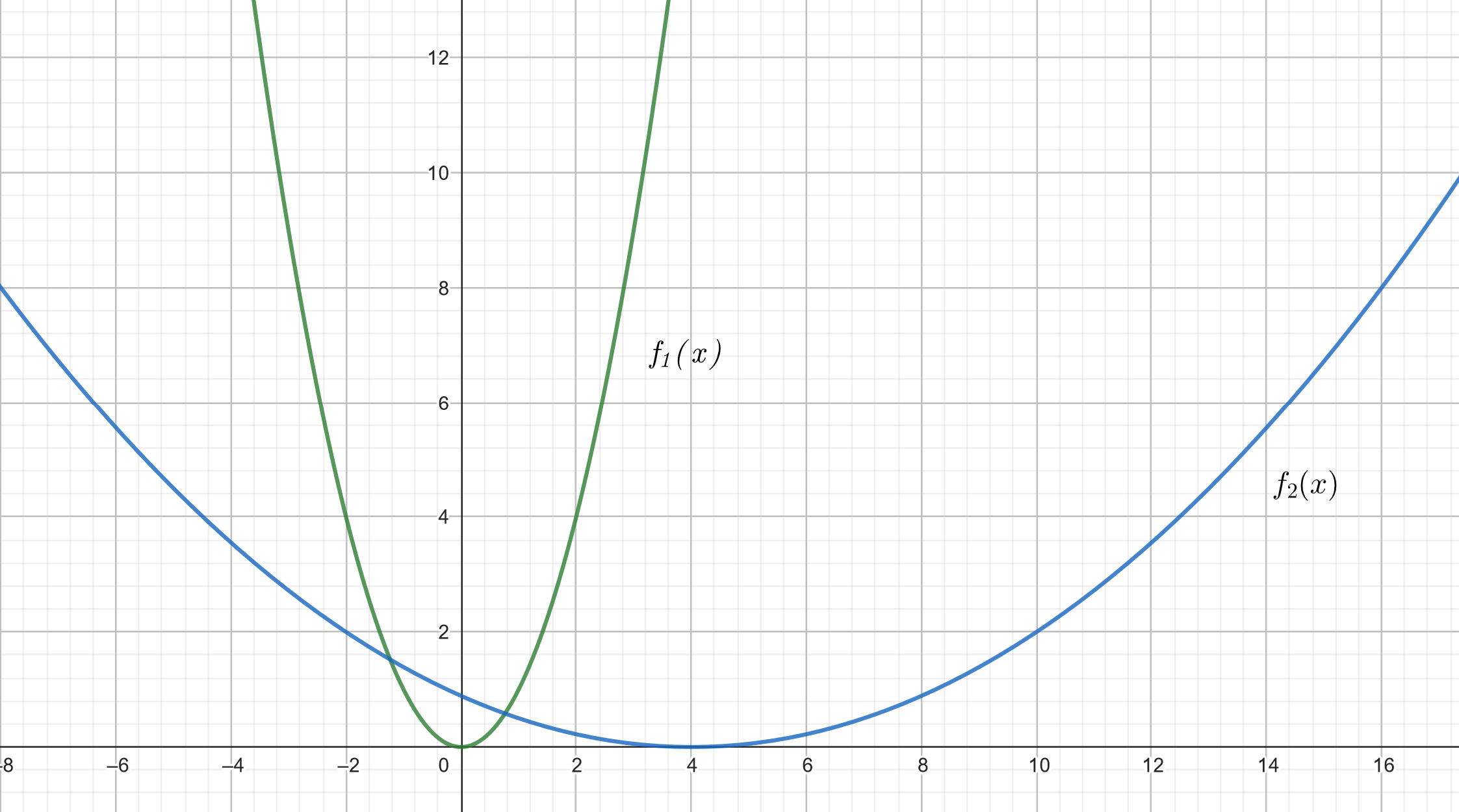}
    \caption{Graphs of the objective functions of Remark \ref{remarkA.1}}
    \label{fig_g1}
\end{figure}

We choose $x_0 = 1$ so that $f^1(x_0) = 1$, $f^2(x_0) = 1/2$. We compute $f_{\max}$ by solving the following problems
\[
f^1_{\max} = \begin{array}[t]{l}
     \max\ x^2 \\
     s.t.\quad \dfrac{(x-4)^2}{18} \leq \dfrac{1}{2} 
\end{array} \qquad f^2_{\max}=
\begin{array}[t]{l}
     \max\ \dfrac{(x-4)^2}{18} \\
     s.t.\quad x^2 \leq 1 
\end{array}
\]
and obtain the reference values $f^1_{\max} = 49$ and $f^2_{\max} = 25/18$. Hence, the initial hypervolume is 
\[
HI(F(\{x_0\})) = 48\times\dfrac{8}{9}.
\]
Now, let us suppose that a new point $x_1 = 0$ is produced where $f(x_1) = (0,16/18)^\top$. As we can see, point $x_1$ sufficiently reduces the first objective function which passes from $f^1(x_0)=1$ to $f^1(x_1) = 0$. As we can also see in figure \ref{fig_g2}, 
the hypervolume increases from $HI(F(\{x_0\}))$ to
\[
HI(F(\{x_0,x_1\})) = HI(F(\{x_0\})) + 1/2
\]
which is less that $1^2$, i.e. the reduction in the first objective function squared.
\begin{figure}
    \centering
    \includegraphics[width=0.5\textwidth]{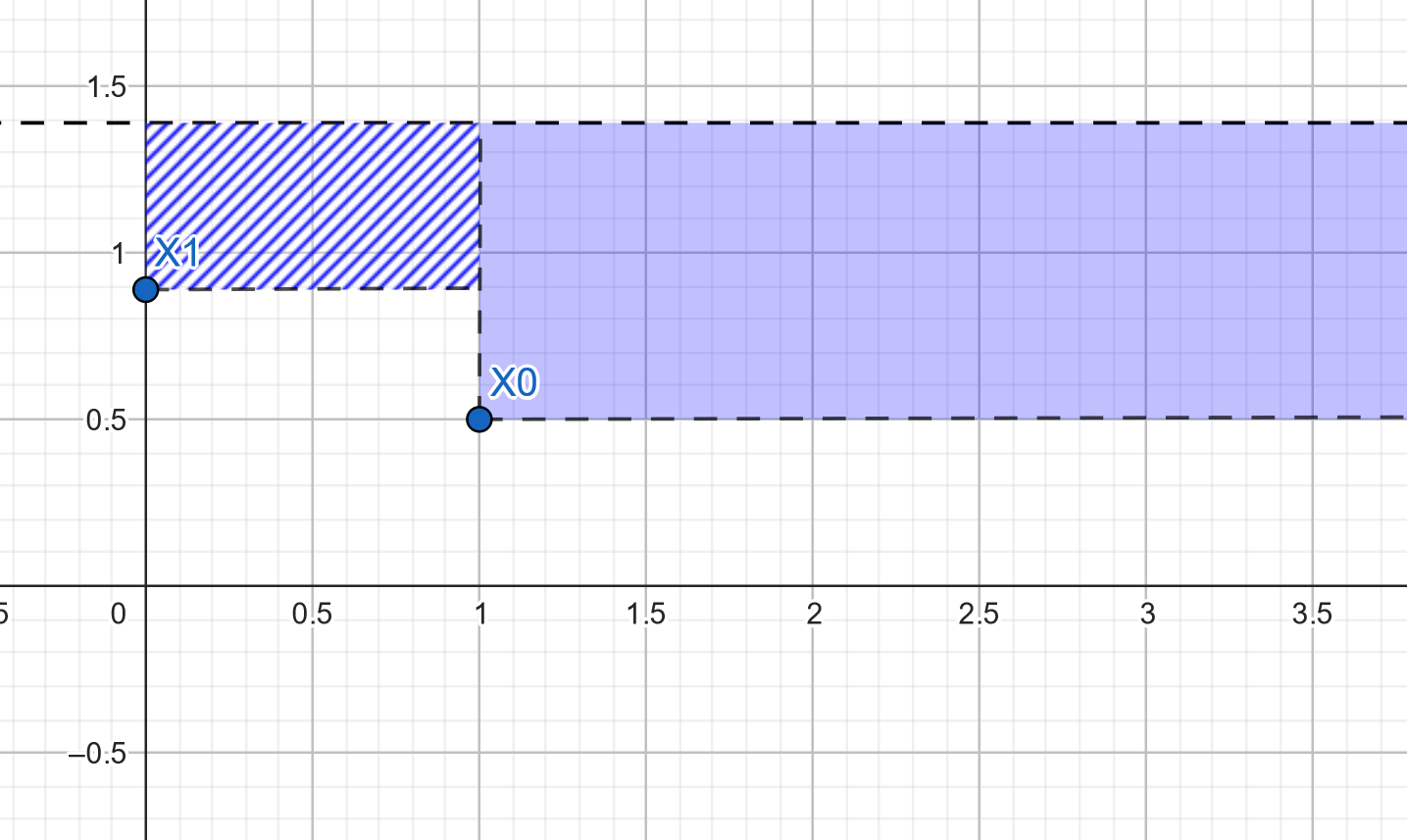}
    \caption{Hypervolume representation for the example in Remark \ref{remarkA.1}}
    \label{fig_g2}
\end{figure}
\end{remark}


\end{document}